\RequirePackage{fix-cm}
\documentclass[smallextended]{svjour3}       % onecolumn (second format)
\smartqed  % flush right qed marks, e.g. at end of proof
\usepackage{graphicx}

\usepackage{lipsum}
\usepackage{amsfonts}
\usepackage{graphicx}
\usepackage{epstopdf}
\usepackage{algorithmic}
\usepackage{amsmath,wasysym,amssymb}
\usepackage{mathtools}
\usepackage{booktabs}
\usepackage{relsize}
\usepackage{float}
\usepackage{subfig}
\usepackage{color}
\usepackage[numbers]{natbib}
\newcommand*{\defeq}{\mathrel{\vcenter{\baselineskip0.5ex \lineskiplimit0pt
			\hbox{\scriptsize.}\hbox{\scriptsize.}}}%
	=}
\newcommand{\Var}{{\mathbb{V}\mathrm{ar}}\,}

\begin{document}

\title{An Application of Fractional Differential Equations to Risk Theory%\thanks{Grants or other notes
}

%\titlerunning{Fractional Equations in Risk Theory}        % if too long for running head

\author{Corina D. Constantinescu        \and
        Jorge M. Ramirez				\and
        Wei R. Zhu
}


\institute{Corina D. Constantinescu and Wei R. Zhu \at
              Institute for Financial and Actuarial Mathematics \\
              University of Liverpool, L69 7ZL, UK\\
              \email{C.Constantinescu@liverpool.ac.uk}           %  \\
           \and
           Jorge M. Ramirez \at
            Universidad Nacional de Colombia, Sede Medellin\\
            Cra 65 59A - 110
            Medellin - Colombia
}

\date{Received: date / Accepted: date}

\maketitle

\begin{abstract}
This paper defines a new class of fractional differential operators alongside a family of random variables whose density functions solve fractional differential equations equipped with these operators. These equations can be further used to construct fractional integro-differential equations for the ruin probabilities in collective renewal risk models, with inter-arrival time distributions from the aforementioned family. Gamma-time risk models and fractional Poisson risk models are two specific cases among them, whose ruin probabilities have explicit solutions, when claim sizes distributions exhibit rational Laplace transforms.



\end{abstract}

\section{Introduction}
The concept of first passage time is widely used in financial mathematics and actuarial science. It could model various things, from the time to dividends payment of a stock, to the exercise date of an American put option, or the ruin probability of an insurance company. In this paper we focus on the \textit{ruin time of an insurance business}, namely the first time in which the business surplus (capital) becomes negative. Our analysis is aimed at solving equations for the probability of ruin expressed as a function of the initial capital (surplus) of the risk process.\\

Motivated by risk theory applications, we consider a new class of risk processes, while extending those from \cite{li2004ruin, albrecher2010algebraic, biard2014fractional} into a fractional derivative framework. It has been proved that ruin probabilities are exponential functions when claim sizes follow an exponential distribution, for various inter-arrival time distributions  \cite{AsmAlb2010}. This paper will derive explicit ruin probabilities in risk models with claim sizes whose distributions have rational Laplace transforms, and inter-arrival time densities solving fractional differential equations. Gamma-time risk models and fractional Poisson risk models are two particular cases among them. All the results are obtained due to the introduction of a new class of fractional differential operators, which extends those from \cite{babenko1986heat, podlubny1998fractional}. These operators generalize  the results from \cite{albrecher2010algebraic} to a fractional derivative framework, in which their explicit results concerning ruin probabilities become particular cases. Some existed ruin probability results are retrieved (see Example \ref{Ex_GammaRM} and \ref{Ex_FracPoisson} for details), and new results are derived. For instance, in the gamma-time risk model with Erlang(2) distributed claim sizes, the ruin probability has the form
\begin{equation*}
A_1 e^{-B_1 u}+A_2 e^{-B_2 u}, \quad u>0,
\end{equation*}
where $A_1$,$B_1$,$A_2$ and $B_2$ are constants that can be calculated on a case-by-case basis (see Example \ref{Gammatimegammaclaim}).\\

The classical collective insurance risk model describes the \textit{surplus} $ R(t) $ of an insurance company over time,
\begin{equation}\label{crm}
R(t)=u+ct-\sum_{i=1}^{N(t)}X_i,\quad t>0
\end{equation}
where $u>0$ is the \textit{initial capital} and $c>0$ is the \textit{premium rate}. The claims occur randomly. The positive random variable $ X_i $ describes the size of the $ i $-th claim, which happened after waiting $ T_i $ units of time since the last claim. The process $ N(t) $ gives the number of claims that have happened up to time $ t $. In the classical model \eqref{crm},  dating back to \cite{lundberg1903approximerad, lundberg1926forsakringsteknisk, Cramer}, all random variables are assumed independent and identically distributed. Moreover, the waiting times are usually assumed to be exponentially distributed, with the resulting counting process $ N(t) $ thus being a Poisson process. The \textit{ruin probability} of this compound Poisson risk model, for an initial capital $u$, is defined as
\begin{equation*}\label{Def_psi}
\psi(u)=\mathbb{P}\left[\left.\inf\{R(t):t>0\}<0\,\right|\,R(0)=u\right].
\end{equation*}
The net profit condition
\begin{equation}\label{Eq_NetProfit}
c\,\mathbb{E}\left[T_i\right]>\mathbb{E}\left[X_i\right]
\end{equation}
is imposed to ensure that ruin does not happen with certainty.
Various generalizations of the classical risk model \eqref{crm} have been considered over time. In \cite{andersen1957collective}, Sparre Andersen defined the renewal risk model. This model accounts for claim number processes $N(t)$ not necessarily Poisson, but verifying the renewal property. The ruin probabilities $\psi(u)$ in renewal models still solve integral equations derived from the renewal property, namely
\begin{align}\label{IEQ}
\psi(u)&=\int_{0}^{\infty}f_T(t)\left(\int_0^{u+ct}\psi(u+ct-y)\,\mathrm{d}F_X(y)+\int_{u+ct}^\infty\,\mathrm{d}F_X(y)\right)\,\mathrm{d}t
\end{align}
with the universal boundary condition $\lim_{u\rightarrow\infty}\psi(u)=0$, as in \cite{feller2008introduction}.  Here $f_T$ and $F_X$ denote the probability density of the waiting time, and the distribution function of the claim size, respectively. This notation  will be used throughout the paper.\\

There is a large actuarial literature analyzing renewal risk processes.  Expressions for the Laplace transform of the ruin probability for risk models with Erlang(2, $\beta$) or mixture  of 2-exponential waiting times were derived in \cite{dickson1998class, dickson1998ruin, dickson2001time} as solutions of second-order differential equations. \cite{lin2000moments} calculated the joint and marginal moments of the time of ruin, the surplus before ruin, and the deficit at ruin, whenever the inter-arrival times distributions have rational Laplace-Stieltjes transform. Subsequently, \cite{dufresne2002general} computed the Laplace transform of the non-ruin probability for inter-arrival times distributions exhibiting rational Laplace transforms. \cite{li2004ruin} used a similar approach as \cite{gerber1998time} to derive a defective renewal equation for the expected discounted penalty due at ruin in a risk model with Erlang($n$) inter-arrival times. Finally, \cite{chen2007ode} derived linear ordinary differential equations for ruin probabilities in Poisson jump-diffusion processes, with phase-type jumps and obtained explicit results in a few instances. The common thread of these paper consists on deriving the ruin probabilities as solutions of (integro-)differential equations. \\

In an attempt to develop a general method, \cite{rosenkranz2008integro,rosenkranz2008solving} introduced two algebraic structures for treating integral operators in conjunction with derivatives, integro-differential operators and integro-differential polynomials. Their method allows the description of the associated differential equations, boundary conditions and solution operators (Green's operator) in a uniform yet formal language. Their algebraic symbolic structures have immediate applications in ruin theory. For instance, as an extension of the Erlang risk model, \cite{albrecher2010algebraic} transformed the integral equation for the expected-discounted-penalty-due-at-ruin function into an integro-differential equation whenever the inter-arrival time distributions have rational Laplace transforms. Rational Laplace transforms densities are equivalent to densities that are solutions of ordinary differential equations with constant coefficients. If the claim size distributions also have rational Laplace transforms, these integro-differential equations can be further reduced to linear boundary value problems. Their symbolic computation approach permits extensions to models with premium dependent on reserves (also discussed in \cite{djehiche1993large} regarding the upper and lower bounds of finite ruin probabilities), the associated boundary problems involving then linear ordinary differential equations with variable coefficients  \cite{albrecher2013exact}. A similar duality idea has been studied in \cite{kolokoltsov2013stochastic} and the reference therein.\\

We show that the probability density function of a sum of independent, heterogeneous gamma and Mittag-Leffler random variables satisfies a fractional differential equation, which we write in an operator/symbolic form. As an application, we consider a family of risk models with inter-arrival times from this family of distributions, and derive the corresponding fractional integro-differential equations satisfied by the corresponding ruin probabilities.
We consider the case of claim sizes described by sums of heterogeneous gamma random variables and show that the corresponding ruin probabilities solve fractional differential equations with constant coefficients. These equations contain both left and right fractional differential operators. We annotate here that these equations can describe other physical phenomena exhibiting anomalous diffusion, as in  \cite{jiao2012existence} where the ``claim sizes" are height losses of the granular material contained in a silo over time \cite{leszczynski2011modeling}. For other applications, we refer to \cite{fix2004least, jiao2011existence, li2013existence, torres2014mountain} and the references therein. 
We also remark that Equation \eqref{timeclaim} presented in this paper can be seen as generalized cases of the fractional boundary problems treated by \cite{jin2016eigenvalue} where critical point theory is used to analyze fractional differential equations with Dirichlet boundary conditions.\\

The gamma-time risk model considered here is the first generalization of the case of Erlang($ n $)-distributed waiting times considered in \cite{li2004ruin}, to that of waiting times distributed as $\Gamma(r, \lambda $), $ r $ being now any positive real number. This is of significance since, in practice, parameter estimation methods usually yield non-integer-valued shape parameters for the gamma distributions that best fit the available data. It thus becomes necessary to study the ruin theory related to real-valued gamma-distributed random variables. \cite{MR0381218} dealt with a special non-integer shape gamma $\Gamma (1/b,1/b)$, $b>1$ distributed claims case, and \cite{cons2017} provided three equivalent expressions for ruin probabilities in a Cram\'er-Lundberg model with gamma distributed claims. Prior to this work, as far as we know, there are no results for non-integer shape gamma-time risk model in the ruin theory literature. The fractional Poisson risk model has been previously treated in \cite{beghin2013large} and \cite{biard2014fractional} for exponential claim sizes, but here, via this fractional calculus approach, we are able to derive expressions for the ruin probability for a larger class of claim sizes in fractional Poisson models.\\

The paper is organized as follows. In Section \ref{operators} we introduce the concept of fractional integro-differential operators. In Section \ref{MainResults} we present the main result and  finally, in Section \ref{2exps},  we perform some illustrative numerical calculations and compare the behavior of the ruin probabilities as a function of the model parameters, for both, the gamma distributed waiting time, and the fractional Poisson risk models. Appendix \ref{bkfc} contains all necessary background on fractional calculus.

%%%%%%%%

\section{Fractional Integro-Differential Operators}\label{operators}

Let $\mathcal{L}(y)$ denote the $n$-th degree polynomial $y^n+p_1y^{n-1}+\dots+p_{n-1}y+p_n$ and consider the following associated homogeneous ordinary differential equation with constant coefficients
\begin{equation}\label{hode}
\mathcal{L}\left(\frac{\mathrm{d}}{\mathrm{d}x}\right)[f](x)=f^{(n)}(x)+p_1f^{(n-1)}(x)+\dots+p_{n-1}f'(x)+p_nf(x)=0.
\end{equation}
%Once the equation \eqref{hode} is fixed, the corresponding polynomial operator $\mathcal{L}\left(\frac{d}{dx}\right)$ is uniquely determined.
Suppose further that Equation \eqref{hode} can be expressed in the form
\begin{equation}\label{hode2}
\bigodot_{j=0}^{m} \left(\frac{\mathrm{d}}{\mathrm{d}x}+\lambda_{j}\right)^{k_{j}}[f](x)=0
\end{equation}
for positive real numbers $ \lambda_j $ and integers $ k_j $, $ j=1,\dots, m $. In \eqref{hode} and henceforth, $\bigodot$ denotes left-composition of operators, namely
\[ \bigodot_{j=1}^{m} \mathcal{L}_j [f] := (\mathcal{L}_m \circ \cdots \circ \mathcal{L}_1)[f].  \]
The solution $f(x)$ to \eqref{hode2} is the probability density function of either a sum of Erlang random variables or a mixed Erlang random variable, depending on the boundary conditions (see \cite{albrecher2010algebraic}). We would like generalize Equation \eqref{hode2}, and characterise its solutions in the case where the exponents $ k_j $ are no longer integers.
\subsection{Left and Right Fractional Differential Operators}
In order to generalize expression \eqref{hode}, it is necessary to explore the world of fractional calculus. Solving fractional differential equations has become an essential issue as fractional-order models appear to be more adequate than previously used integer-order models in various fields. A large number of available analytical methods for solving fractional order integral and differential equations is discussed in \cite{podlubny1998fractional}, including the Mellin transform method, the power series method, and the symbolic method. 

The symbolic method was first introduced in \cite{babenko1986heat} and generalizes the Laplace transform method: it uses a specific expansion (e.g., binomial or geometric) on the differential operator and writes it as an infinite sum of fractional derivatives. However, it is always necessary to check the validity of the formal expansion since the interchange of infinite summation and integration requires justification. It is nevertheless a powerful tool for determining the possible form of the solution. Numerous examples of the application of this method to heat and mass transfer problems are discussed in \cite{babenko1986heat}.

In this section we define a new family of operators based on the binomial expansion. All of the related definitions and propositions of fractional calculus can be found in Appendix \ref{bkfc}. The important motivation underlying the following definition comes from realising that for positive integer $ n $ and $ \alpha \in \mathbb{R} $,
\begin{equation*}
\left(\frac{\mathrm{d}}{\mathrm{d}x}+\alpha\right)^n [f](x) = e^{-\alpha x} \frac{\mathrm{d}^n}{\mathrm{d}x^n} \left[e^{\alpha x} f(x)\right],
\end{equation*}
and similarly for $ (-\frac{\mathrm{d}}{\mathrm{d}x}+\alpha)^n $. We thus define the following operators as the natural generalization in terms of fractional derivatives:
\begin{definition}\label{RockOp}
	Let $r>0$, $\alpha\in\mathbb{R}$, $a\in[-\infty,\infty)$ and $b\in(-\infty,\infty]$.  The \emph{left fractional differential operator (LFDO)} $\prescript{\alpha}{a} {\mathrm{R}}_x^r$ is defined by 
	\begin{equation}
	\prescript{\alpha}{a} {\mathrm{R}}_x^r \left[f\right](x)\defeq e^{-\alpha x}\prescript{}{a} {\mathrm{D}}_x^{r}\left[e^{\alpha x}\,f(x)\right]
	\end{equation}
	and the \emph{right fractional differential operator (RFDO)} $\prescript{\alpha}{x} {\mathrm{R}}_b^r$ by 
	\begin{equation}\label{RFDO}
	\prescript{\alpha}{x} {\mathrm{R}}_b^r \left[g\right](x)\defeq e^{\alpha x}\prescript{\mathrm{C}}{x}{\mathrm{D}}_b^{r}\left[e^{-\alpha x}\,g(x)\right].
	\end{equation}
	The domain of definition of $\prescript{\alpha}{a} {\mathrm{R}}_x^r$ and $\prescript{\alpha}{x} {\mathrm{R}}_b^r$ are those of the left Riemann-Liouville fractional derivatives $\prescript{}{a}{\mathrm{D}}_x^{r}$ and right Caputo fractional derivatives $\prescript{\mathrm{C}}{x}{\mathrm{D}}_b^{r}$ respectively, which are given in Definition \ref{Def_RLD} and Definition \ref{Def_Caputo}.
\end{definition}

In the case $ a=0 $, integration by parts yields the following characterisation of the formal adjoint of  $\prescript{\alpha}{0} {\mathrm{R}}_x^r$. Along with the integration by parts formula in Proposition \ref{PropFIBP}, this is the key calculation needed for the proof of our main result.

\begin{proposition}\label{Prop_AdjointR}
	Let $ \alpha \in \mathbb{R} $ and $ r>0 $. The formal adjoint with respect to integration by parts of the LFDO $\prescript{\alpha}{0} {\mathrm{R}}_x^r$ is the RFDO $\prescript{\alpha}{x} {\mathrm{R}}_\infty^r$, namely,
	$$
	\int_0^\infty \prescript{\alpha}{0} {\mathrm{R}}_x^r [f](x)\,g(x)\, \mathrm{d}x=\int_0^\infty f(x)\,\prescript{\alpha}{x} {\mathrm{R}}_\infty^r[g](x)\,\mathrm{d}x,
	$$
	for appropriate functions $f$ and $g$ (see Proposition \ref{PropFIBP}).
\end{proposition}

Note that the LFDO can be used to construct differential equations for probability density functions. Consider a gamma probability density function with shape parameter $r\in\mathbb{R}^+$ and rate parameter $\lambda\in\mathbb{R}^+$, namely
\begin{equation*}
f_r(x) =\frac{\lambda^{r}}{\Gamma(r)}x^{r-1}e^{-\lambda x}, \quad x>0.
\end{equation*}

When $r$ is not an integer, instead of an ordinary differential equation, the gamma density function solves the fractional differential equation
\begin{equation}
\label{de}
\prescript{\lambda}{0} {\mathrm{R}}_x^{r}[f_r](x)=e^{-\lambda x}\,\prescript{}{0} {\mathrm{D}}_x^{r}\left[e^{\lambda x}\,f_r(x)\right]=0, \quad x>0,
\end{equation}
with boundary conditions $\prescript{\lambda}{0} {\mathrm{R}}_x^{r-1}[f_r](0)=\lambda^r$ and $\prescript{\lambda}{0} {\mathrm{R}}_x^{r-k}[f_r](0)=0$ for $k=2,\dots,\lceil{r}\rceil$. Another distribution related to the LFDO is the Mittag-Leffler distribution, which is the waiting time distribution in the fractional Poisson process (see in Appendix \ref{rfpp}). The Mittag-Leffler probability density function with parameter $\mu\in (0,1]$ and $\lambda\in\mathbb{R}^+$ is
\begin{equation*}
f_\mu(x) =\lambda x^{\mu-1}E_{\mu,\mu}(-\lambda x^\mu),\quad t>0,
\end{equation*}
and solves the following fractional differential equation
\begin{equation}\label{dde}
\left(\prescript{0}{0} {\mathrm{R}}_x^{\mu}+\lambda\right)[f_\mu](x)=(\prescript{}{0} {\mathrm{D}}_x^{\mu}+\lambda)[f_\mu](x)=0, \quad x>0,
\end{equation}
with the boundary condition $\prescript{}{0} {\mathrm{D}}_x^{\mu-1}[f](0)=\lambda$. Here, the function $E_{\mu,\mu}$ is called two-parameter Mittag-Leffler function, which is defined in Equation \eqref{TPML}.

\subsection{A generalized family of random variables}

The next theorem introduces the family of random variables to which the approach presented in this paper applies to. In its full generality, we consider random variables that can be written as finite sums of independent heterogeneous gamma and Mittag-Leffler random variables. At the moment, there is no known explicit formula for the probability density function of such a random variable, but one can always express it in a convolution form. Notice that if only gamma random variables with integer shape parameters are involved in the summation, this random variable is the generalized integer gamma distribution (GIG) \cite{coelho1998generalized}. We now characterise the fractional boundary value problem satisfied by the density function of such random variables.

\begin{theorem}\label{tMLGamma}
	Consider a random variable $T$ defined by
	\begin{equation}\label{Eq_Tsum}
	T=\sum_{i=1}^{m}Y_i+\sum_{j=1}^{n}Z_j,
	\end{equation}
	in terms of gamma random variables $Y_i\sim \Gamma(r_i,\lambda_{1,i})$ and Mittag-Leffler random variables $Z_j\sim \mathrm{ML}(\mu_j,\lambda_{2,j})$, all independent of each other. Here $r_i$, $\lambda_{1,i}$, $\lambda_{2,j}\in\mathbb{R}^+$ and $\mu_j\in(0,1]$. Then the density function $f_{T}^{m,n}(t)$ of $T$ solves the following fractional differential equation
	\begin{equation}\label{MLGamma}
	\mathcal{A}_{m,n}\left(\frac{\mathrm{d}}{\mathrm{d}t}\right)\left[f_{T}^{m,n}\right](t)\defeq\bigodot_{j=1}^{n}\left(\prescript{}{0} {\mathrm{D}}_t^{\mu_j}+\lambda_{2,j}\right) \bigodot_{i=1}^m \prescript{\lambda_{1,i}}{0} {\mathrm{R}}_t^{r_i}\left[f_{T}^{m,n}\right](t)=0,
	\end{equation}
	with boundary conditions (when $n\neq 0$)
	\begin{align*}
	&\prescript{}{0} {\mathrm{D}}_t^{\mu_1-1}\bigodot_{j=2}^{n}\left(\prescript{}{0} {\mathrm{D}}_t^{\mu_j}+\lambda_{2,j}\right) \bigodot_{i=1}^m \prescript{\lambda_{1,i}}{0} {\mathrm{R}}_t^{r_i}\left.[f_{T}^{m,n}](t)\right|_{t=0}=\Lambda_{m,n},\\
	\text{and}\quad\quad
	&\prescript{}{0} {\mathrm{D}}_t^{\mu_1-k}\bigodot_{j=2}^{n}\left(\prescript{}{0} {\mathrm{D}}_t^{\mu_j}+\lambda_{2,j}\right) \bigodot_{i=1}^m \prescript{\lambda_{1,i}}{0} {\mathrm{R}}_t^{r_i}\left.[f_{T}^{m,n}](t)\right|_{t=0}=0,
	\end{align*}
	for $k=2,\dots,\lceil \sum_{j=1}^n \mu_j +\sum_{i=1}^m r_i \rceil.$ Here and subsequently $\Lambda_{m,n}$ denotes
	\begin{equation}\label{Def_Lambda}
	\Lambda_{m,n} \defeq \prod\limits_{i=1}^m \lambda_{1,i}^{r_i}\prod\limits_{j=1}^n \lambda_{2,j}.
	\end{equation}
\end{theorem}

\begin{proof}
	We defer the proof of Theorem \ref{tMLGamma} to Appendix \ref{appB}.
\end{proof}

\begin{remark}
	We further assume that all $\lambda_{1,i}$  are different, i.e., $\lambda_{1,i}\neq \lambda_{1,k}$ for all $i\neq k$. In other words, each variable $Y_i$ has the gamma distribution with different rate parameters. The uniqueness of the $\lambda_{1,i}$, rate parameter of the gamma random variable could be realized without any loss of generality. Whenever we have $\lambda_{1,i}=\lambda_{1,k}$, $i\neq k$, we would consider the sum of their corresponding random variables, which is still a gamma random variable. 
\end{remark}

\begin{remark}\label{kbc}
	One can show that the boundary conditions in Theorem \ref{tMLGamma} have various equivalent expressions. For any positive integer number $ k\leqslant\lceil \sum_{i=1}^m r_i+\sum_{j=1}^n \mu_j\rceil,$ by choosing non-negative integers $k_{1,i}$ and $k_{2,j}$ such that
	$\sum_{i=1}^m k_{1,i} +\sum_{j=1}^n k_{2,j} =k,$
	we have the boundary conditions of Equation \eqref{MLGamma} as
	\[  \left(\bigodot_{j=1}^{n}\left( \prescript{}{0} {\mathrm{D}}_t^{\mu_j-k_{2,j}}+\lambda_{2,j}\prescript{}{0} {\mathrm{I}}_t^{k_{2,j}}\right)\bigodot_{i=1}^m \prescript{\lambda_{1,i}}{0} {\mathrm{R}}_t^{r_i-k_{1,i}}\right)\left.[f_T^{m,n}](t)\right|_{t=0}= \begin{cases*}
	\Lambda_{m,n},\,k=1\\
	\quad\\
	0,\quad\,\,\,\, k>1.
	\end{cases*} \]
\end{remark}

\begin{remark}
	Equation \eqref{MLGamma} along with its boundary conditions can be regarded as the generalization of a pair of boundary problems discussed in \cite{rosenkranz2008solving}. When the fractional differential algebra is properly defined, these fractional-order boundary problems can be factorised and further solved by obtaining their corresponding Green's operators.
\end{remark}

The solution to Equation \eqref{MLGamma} depends on the boundary condition. When different boundary conditions are given, we may obtain density functions for other possible random variables. For instance, let us consider the differential equation 
$$\left(\frac{\mathrm{d}}{\mathrm{d}t}+\lambda\right)^2 f^{2,0}_T(t)=0$$
with two distinct sets of boundary conditions. First, if we impose
$$
\begin{cases}
\left(\frac{\mathrm{d}}{\mathrm{d}x}+\lambda\right) f^{2,0}_T(t)|_{t=0}=\lambda^2,\\
\quad\\
\lambda f^{2,0}_T(t)|_{t=0}=0,
\end{cases}
$$
the solution is the Elang($2,\lambda$) density function $f^{2,0}_T(t)=\lambda^2te^{-\lambda t}$ which belongs to the random variable family considered in Equation \eqref{Eq_Tsum}. However, the solution to the above equation would become $f^{2,0}_T(t)=\frac{1}{2}\lambda e^{-\lambda t}+\frac{1}{2}\lambda^2te^{-\lambda t}$ if the boundary conditions are changed to
$$
\begin{cases}
\left(\frac{\mathrm{d}}{\mathrm{d}x}+\lambda\right)f^{2,0}_T(t)|_{t=0}=\frac{1}{2} \lambda^2,\\
\quad\\
\lambda f^{2,0}_T(t)|_{t=0}=\frac{1}{2} \lambda^2.
\end{cases}
$$
This solution is the density function of a mixture of an exponential and an Erlang random variable, and the associated distribution does not satisfy Equation \eqref{Eq_Tsum}.

\section{Main Results}\label{MainResults}
The LFDO and RFDO give us the ability to study a very general family of distributions that may find applications in various areas, e.g, queuing theory, risk theory and control theory. Although many of the available techniques for the analysis of the associated equations are numerical or asymptotic, the fractional differential approach still offers analytic insights to the related problems. In this section, we aim at accomplishing this with particular problems arising in risk theory. A special family of renewal risk models of the form \eqref{crm} will be considered, including the Erlang$(n)$ and fractional Poisson risk models. We will show that the ruin probabilities in these models solve fractional integro-differential equations involving the LFDO and RFDO operators.

Before moving on to the main result, we introduce a lemma that allows us to change the argument of our operators on a bivariate function under certain circumstances.

\begin{lemma}\label{bivari}
	For positive real numbers $ \alpha$, $r$ and  $c$, the following identity holds
	\begin{equation}\label{ROCV}
	\prescript{\alpha}{x} {\mathrm{R}}_\infty^r[f(x+cy)](x,y)=c^{-r}\prescript{\alpha c}{y} {\mathrm{R}}_\infty^r[f(x+cy)](x,y),
	\end{equation}
	where $x$ and $y$ are real numbers and $\prescript{\alpha}{x} {\mathrm{R}}_\infty^r$ is defined in Equation \eqref{RFDO}.
\end{lemma}

\begin{proof}
	We start from the left-hand side of Equation \eqref{ROCV}. By definition we have
	\begin{align*}
	&\prescript{\alpha}{x} {\mathrm{R}}_\infty^r[f(x+cy)](x,y)=e^{\alpha x} \frac{1}{\Gamma(n-r)}\int_x^\infty(t-x)^{n-r-1}\frac{\mathrm{d}^n}{\mathrm{d}t^n}\left[e^{-\alpha t}f(t+cy)\right]\,\mathrm{d}t.
	\end{align*}
	Letting $s=\frac{1}{c}(t-x)+y$ leads to
	\begin{align*}
	\frac{1}{\Gamma(n-r)}\int_y^\infty e^{\alpha cy}(s-y)^{n-r-1}c^{-r}\frac{\mathrm{d}^n}{\mathrm{d}y^n}\left[e^{-\alpha cs}f(cs+x)\right]\,\mathrm{d}s,
	\end{align*}
	which is the right-hand side of Equation \eqref{ROCV}.
\end{proof}

Now we are able to generalize the result from \cite{li2004ruin, albrecher2010algebraic, biard2014fractional} to a risk model with inter-arrival times of the form of \eqref{Eq_Tsum}.
The main result of this paper is the following:
\begin{theorem}\label{mainthm}
	Consider a renewal risk model
	\begin{equation*}
	R_{m,n}(t)=u+ct-\sum_{i=1}^{N_{m,n}(t)}X_i,\quad t>0,
	\end{equation*}
	where the inter-arrival times $T_k$ are assumed to be a finite sum of independent gamma random variables $Y_i\sim\Gamma(r_i,\lambda_{1,i})$ and Mittag-Leffler random variables $Z_j\sim \mathrm{ML}(\mu_j,\lambda_{2,j})$ as in \eqref{Eq_Tsum}. Then the ruin probability $\psi(u)$ under model $ R_{m,n} $, satisfies the following fractional integro-differential equation
	\begin{align}\label{maineq}
	&\mathcal{A}_{m,n}^{*}\left(c\frac{\mathrm{d}}{\mathrm{d}u}\right)[\psi](u)
	=\Lambda_{m,n}\left(\int_{0}^{u}\psi(u-y)\,\mathrm{d}F_X(y)+\int_{u}^{\infty}\,\mathrm{d}F_X(y)\right)
	\end{align}
	with the universal boundary condition $\lim_{u\rightarrow\infty}\psi(u)=0$. Here, the constant $ \Lambda_{m,n} $ is given by Equation \eqref{Def_Lambda} and $ \mathcal{A}_{m,n}^{*} $ is the formal adjoint of $ \mathcal{A}_{m,n} $ (see Equation \eqref{MLGamma}) and is given by
	\begin{equation}\label{Def_AmnStar}
	\mathcal{A}_{m,n}^{*}\left(c\frac{\mathrm{d}}{\mathrm{d}u}\right) \defeq
	\bigodot_{j=1}^{n}\left(c^{\mu_j}\prescript{\mathrm{C}}{u} {\mathrm{D}}_\infty^{\mu_j}+\lambda_{2,j}\right) \bigodot_{i=1}^m \left(c^{r_i}\prescript{\lambda_{1,i}/c}{u} {\mathrm{R}}_\infty^{r_i}\right).
	\end{equation} 
	
\end{theorem}

\begin{proof}
	For a general renewal risk model, the ruin probability solves the renewal equation \eqref{IEQ}  (see \cite{feller2008introduction}). Denoting the terms in parentheses of \eqref{IEQ} by
	$$h(u+ct)=\int_{0}^{u+ct}\psi(u+ct-y)\,\mathrm{d}F_X(y)+\int_{u+ct}^{\infty}\,\mathrm{d}F_X(y),$$
	we now apply $\mathcal{A}_{m,n}^{*}(c\frac{\mathrm{d}}{\mathrm{d}u})$ on both sides of the renewal equation and use Lemma \ref{bivari} to obtain
	\begin{align*}
	&\mathcal{A}_{m,n}^{*}\left(c\frac{\mathrm{d}}{\mathrm{d}u}\right)[\psi](u)
	=\int_0^\infty f_T^{m,n}(t)  \mathcal{A}_{m,n}^{*}\left(\frac{\mathrm{d}}{\mathrm{d}t}\right)[h(u+ct)](u,t)\,\mathrm{d}t.
	\end{align*}
	The fractional integration by parts rule in Equation \eqref{FIBP} is applicable here,
	\begin{align*}
	&\,\,\,\,\,\,\,\int_0^\infty f^{m,n}_T(t) \mathcal{A}_{m,n}^{*}\left(\frac{\mathrm{d}}{\mathrm{d}t}\right)[h(u+ct)](u,t)\,\mathrm{d}t\\
	&=\int_0^\infty\left(\prescript{}{0} {\mathrm{D}}_t^{\mu_1}+\lambda_{2,1} \right)[f_T^{m,n}](t)\mathcal{A}_{m,n-1}^{*}\left(\frac{\mathrm{d}}{\mathrm{d}t}\right)[h(u+ct)](u,t)\,\mathrm{d}t+\\
	&\sum_{k=0}^{\lfloor \mu_{1}\rfloor}\left[(-1)^{\lfloor \mu_{1}\rfloor+1+k}\prescript{}{0}{\mathrm{D}}_t^{\mu_1+k-\lfloor \mu_1\rfloor-1}[f_T^{m,n}](t)
	\left.\mathcal{A}_{m,n-1}^{*}\left(\frac{\mathrm{d}}{\mathrm{d}t}\right)[h(u+ct)](u,t)\right|_0^\infty\right].
	\end{align*}
	The boundary condition term evaluated at $t=0$ could be computed by using the initial value theorem of Laplace transforms,
	\begin{align*}
	&\prescript{}{0}{\mathrm{I}}_t^{1-\mu_1}[f_T^{m,n}](0)
	%=\lim_{s\rightarrow \infty}s\int_0^\infty e^{-st}\prescript{}{0}{I}_t^{1-\mu_1}[f_T^{m,n}](t)\,dt
	=\lim_{s\rightarrow \infty}\left(s^{\mu_1}\prod_{j=1}^n \frac{\lambda_{2,j}}{s^{\mu_j}+\lambda_{2,j}}\prod_{i=1}^m \left(\frac{\lambda_{1,i}}{s+\lambda_{1,i}}\right)^{r_i}\right)=0.
	\end{align*}
	Another boundary condition term evaluated at $t=\infty$ also equals  zero due to the fact that the definition of the right Caputo fractional derivative is an integral from $t$ to $\infty$. Analogously, we are able to move the first $n$ operators $\bigodot_{j=1}^{n}(\prescript{\mathrm{C}}{t} {\mathrm{D}}_\infty^{\mu_j}+\lambda_{2,j})$ from function $h$ to $f_T^{m,n}$ with all boundary conditions vanishing, which leads to
	\begin{align*}
	&\mathcal{A}_{m,n}^{*}\left(c\frac{\mathrm{d}}{\mathrm{d}u}\right)[\psi](u)\\
	=& \int_0^\infty\bigodot_{j=1}^{n}\left(\prescript{}{0} {\mathrm{D}}_t^{\mu_j}+\lambda_{2,j}\right) [f_T^{m,n}](t) \bigodot_{i=1}^m \prescript{\lambda_{1,i}}{t} {\mathrm{R}}_\infty^{r_i}[h(u+ct)](u,t)\,\mathrm{d}t.
	\end{align*}
	Now we use the integration by parts formula in Proposition \ref{Prop_AdjointR} to take the first RFDO $\prescript{\lambda_{1,1}}{t} {\mathrm{R}}_\infty^{r_1}$ off of $h$. Furthermore it can be shown that its adjoint $\prescript{\lambda_{1,1}}{0} {\mathrm{R}}_t^{r_1}$ commutes with  $(\prescript{}{0} {\mathrm{D}}_t^{\mu_j}+\lambda_{2,j})$ for all $j=1,\dots, n$ when applied on the density function $f_T^{m,n}$. We therefore get the right-hand side equal to:
	
	\begin{align*}
	%\mathcal{A}_{m,n}^{*}\left(c\frac{d}{du}\right)[\psi](u)=
	&\int_0^\infty\bigodot_{j=1}^{n}\left(\prescript{}{0} {\mathrm{D}}_t^{\mu_j}+\lambda_{2,j}\right) \prescript{\lambda_{1,1}}{0} {\mathrm{R}}_t^{r_1} [f_T^{m,n}](t)\bigodot_{i=2}^m \prescript{\lambda_{1,i}}{t} {\mathrm{R}}_\infty^{r_i}[h(u+ct)](u,t)\,\mathrm{d}t\\
	&+\sum_{k=0}^{\lfloor r_{1}\rfloor}\left[(-1)^{\lfloor r_{1}\rfloor+1+k} \bigodot_{i=2}^m \prescript{\lambda_{1,i}}{t} {\mathrm{R}}_\infty^{r_i}[h(u+ct)](u,t)\right.\\
	&\quad \quad \quad \left.\left. \bigodot_{j=1}^{n}\left(\prescript{}{0} {\mathrm{D}}_t^{\mu_j}+\lambda_{2,j}\right) \prescript{\lambda_{1,i}}{0} {\mathrm{R}}_t^{r_1+k-\lfloor r_{1}\rfloor-1} [f_T^{m,n}](t)\right]\right|_0^\infty.
	\end{align*}
	The boundary condition at $t=0$ can be computed by applying the initial value theorem
	\begin{align*}
	&\bigodot_{j=1}^{n}\left(\prescript{}{0} {\mathrm{D}}_t^{\mu_j}+\lambda_{2,j}\right) \prescript{\lambda_{1,1}}{0} {\mathrm{R}}_t^{r_1+k-\lfloor r_{1}\rfloor-1} [f_T^{m,n}](0)\\
	%=\prod_{j=1}^n \lambda_{2,j}  \left(e^{-\lambda_{1,1} t}\prescript{}{0} {D}_t^{r_1+k-\lfloor r_{1}\rfloor-1} \left(e^{\lambda_{1,1} t} f_T^{m,0}(t)\right)\right)\Big |_{t=0}\\
	%=&\prod_{j=1}^n \lambda_{2,j} \lim_{s\rightarrow \infty}\left(s\int_{0}^\infty e^{-(s+\lambda_{1,1}) t}\prescript{}{0} {D}_t^{r_1+k-\lfloor r_{1}\rfloor-1} \left(e^{\lambda_{1,1} t} f_T^{m,0}(t)\right)\,dt\right)\\
	=&\prod_{j=1}^n \lambda_{2,j} \lim_{s\rightarrow \infty}\left(\frac{\lambda_{1,1}^{r_1} s}{(s+\lambda_{1,1})^{\lfloor r_{1}\rfloor+1-k}}\prod_{i=2}^m \left(\frac{\lambda_{1,i}}{s+\lambda_{1,i}}\right)^{r_i}\right.\\
	&-s\left.\sum_{l=0}^{k-1}(s+\lambda_{1,1})^l \left[ \prescript{}{0} {\mathrm{D}}_t^{r_1+k-\lfloor r_{1}\rfloor-l-2}\left(e^{\lambda_{1,1}}f_T^{m,0}(t)\right) \right]\Big |_{t=0}\right).
	\end{align*}
	We continue to iteratively use the initial value theorem on the terms $$s(s+\lambda_{1,1})^l \left[ \prescript{}{0} {\mathrm{D}}_t^{r_1+k-\lfloor r_{1}\rfloor-l-2}\left(e^{\lambda_{1,1} t}f_T^{m,0}(t)\right) \right]\Big |_{t=0},$$ until it eventually gives us
	\begin{align*}
	&s(s+\lambda_{1,1})^{\lfloor r_{1}\rfloor-1} \left[ \prescript{}{0} {\mathrm{I}}_t^{\lfloor r_{1}\rfloor+1-r_1}\left(e^{\lambda_{1,1} t}f_T^{m,0}(t)\right) \right]\Big |_{t=0}\\
	=&s(s+\lambda_{1,1})^{r_1-2}\prod_{i=1}^{m}\left(\frac{\lambda_{1,i}}{s}\right)^{r_i},
	\end{align*}
	which tends to zero when $s\rightarrow\infty$. The boundary condition term evaluated at $t=\infty$ yields zero since the right Caputo derivatives vanish at infinity. Analogously, we are able to move the rest of the operators $\bigodot_{i=1}^m \prescript{\lambda_{1,i}}{t} {\mathrm{R}}_\infty^{r_i}$ from the function $h$ to $f_T^{m,n}$ with all boundary conditions vanishing, which leads to
	\begin{align*}
	\mathcal{A}_{m,n}^{*}\left(c\frac{\mathrm{d}}{\mathrm{d}u}\right)[\psi](u) &=\int_0^\infty\mathcal{A}_{m,n}\left(\frac{\mathrm{d}}{\mathrm{d}t}\right) [f_T^{m,n}](t) [h(u+ct)](u,t)\,\mathrm{d}t\\
	&+\left(\left.[h(u+ct)](u,t) \mathcal{A}_{m-1,n}\left(\frac{\mathrm{d}}{\mathrm{d}t}\right) [f_T^{m,n}](t)\right|_{t=0}\right).
	\end{align*}
	Since the inter-arrival time density satisfies Equation \eqref{MLGamma}, the integral term of the above equation vanishes. The boundary conditions of $f_T^{m,n}$ ensure that the lower summand is, at $ t=0 $,
	$$
	h(u) \bigodot_{j=1}^{n}\left(\prescript{}{0} {\mathrm{D}}_t^{\mu_j}+\lambda_{2,j}\right) \prescript{\lambda_{1,n}}{0} {\mathrm{R}}_t^{r_n-1} \bigodot_{i=1}^{m-1}\prescript{\lambda_{1,i}}{0} {\mathrm{R}}_t^{r_i}[f_T^{m,n}](0)=\Lambda_{m,n} h(u)
	$$
	This completes the proof.
\end{proof}

\begin{corollary}
	The non-ruin probability $\phi(u) = 1-\psi(u)$ for the risk model in Theorem \ref{mainthm} satisfies the following fractional integro-differential equation
	\begin{align}\label{maineq2}
	&\mathcal{A}_{m,n}^{*}\left(c\frac{\mathrm{d}}{\mathrm{d}u}\right)[\phi](u)
	=\Lambda_{m,n}\left[\int_{0}^{u}\phi(u-y)\,\mathrm{d}F_X(y)\right]
	\end{align}
	with the universal boundary condition $\lim_{u\rightarrow\infty}\phi(u)=1$ (see Equation \eqref{Def_Lambda} and \eqref{Def_AmnStar} for the definitions of the constant $ \Lambda_{m,n} $ and the operator $ \mathcal{A}_{m,n}^{*}(c\frac{\mathrm{d}}{\mathrm{d}u}) $).
\end{corollary}

Theorem \ref{mainthm} characterises a fractional integro-diferential equation satisfied by the ruin probability $ \psi $ for a large class of waiting time distributions. The solvability of this fractional integro-differential equation depends on the particular form of the claim size distribution function $ F_X $. 

We now restrict the rest of the analysis to claim sizes $X_i$ distributed as a sum of an arbitrary number of independent gamma random variables. The next theorem shows that under such assumption, Equation \eqref{maineq} can be written as a boundary value problem with only fractional derivatives. It is important to note that if the claim sizes include any Mittag-Leffler components, as it is the case of $T$ in Theorem \ref{mainthm}, we would have $\mathbb{E}[X_i] = \infty$ and ruin would happen with probability one since the net profit condition is violated.

\begin{theorem}\label{mainthm2}
	Consider the renewal risk model in Theorem \ref{mainthm}. Assume further that the claim sizes $X_i$ are each distributed as a sum of $l$ independent $\Gamma(s_k,\alpha_k)$ distributed random variables for some $ s_k,\alpha_k >0 $, $ k=1,\dots,l $ i.e., $f_X$ satisfies
	\begin{equation}\label{Def_Al}
	\mathcal{A}_{l}\left(\frac{\mathrm{d}}{\mathrm{d}u}\right)\left[f_{X}\right](u) \defeq \bigodot_{k=1}^l \prescript{\alpha_k}{0} {\mathrm{R}}_u^{s_k}\left[f_{X}\right](u)=0,
	\end{equation}
	with certain boundary conditions (see Theorem \ref{tMLGamma}). Let $ \mathcal{A}_{m,n}^*(c\frac{\mathrm{d}}{\mathrm{d}u}) $ and $ \Lambda_{m,n} $ be as defined in Equation \eqref{Def_AmnStar} and \eqref{Def_Lambda} respectively. Then the non-ruin probability $\phi(u)$ satisfies 
	\begin{align}\label{timeclaim} 
	&\mathcal{A}_{l}\left(\frac{\mathrm{d}}{\mathrm{d}u}\right)\mathcal{A}_{m,n}^*\left(c\frac{\mathrm{d}}{\mathrm{d}u}\right)[\phi](u)
	=\Lambda_{m,n} \prod_{k=1}^l \alpha_k^{s_k} 
	\phi(u)
	\end{align}
	with the universal boundary condition $\lim_{u\rightarrow\infty}\phi(u)=1$ and initial values
	\begin{equation}\label{timeclaimBC}
	\left(\prescript{\alpha_1}{0} {\mathrm{R}}_u^{s_1-k'}\bigodot_{k=2}^l \prescript{\alpha_k}{0} {\mathrm{R}}_u^{s_k}\bigodot_{j=1}^{n}\left(c^{\mu_j}\prescript{\mathrm{C}}{u} {\mathrm{D}}_\infty^{\mu_j}+\lambda_{2,j}\right) \bigodot_{i=1}^m \left(c^{r_i}\prescript{\lambda_{1,i}/c}{u} {\mathrm{R}}_\infty^{r_i}\right)\right)[\phi](0)=0,
	\end{equation}
	for $k'=1,\,\dots,\,\left\lceil \sum_{k=1}^l s_k \right\rceil-1.$ 
\end{theorem}

\begin{proof}
	Taking the operator $\mathcal{A}_{l}(\frac{\mathrm{d}}{\mathrm{d}y})$ on both sides of Equation \eqref{maineq2} leads to
	\begin{align*}
	&\mathcal{A}_{l}\left(\frac{\mathrm{d}}{\mathrm{d}u}\right)\mathcal{A}_{m,n}^*\left(c\frac{\mathrm{d}}{\mathrm{d}u}\right)[\phi](u)
	= \Lambda_{m,n} \mathcal{A}_{l}\left(\frac{\mathrm{d}}{\mathrm{d}u}\right)\left(\int_{0}^{u}\phi(u-y)f_X(y)\,\mathrm{d}y\right).
	\end{align*}
	Recall from Theorem \ref{tMLGamma} that the non-ruin probability function $\phi(u)$ is supported on $[0,\infty)$, so the identity
	\begin{equation*}
	\mathcal{A}_{l}\left(\frac{\mathrm{d}}{\mathrm{d}u}\right)\left(\int_{0}^{u}\phi(u-y)f_X(y)\,\mathrm{d}y\right)=\bigodot_{k=1}^l \prescript{\alpha_k}{0} {\mathrm{R}}_u^{s_k}[\phi*f_{X}](u)=\prod_{k=1}^l \alpha^{s_k} \phi(u)
	\end{equation*}
	holds in this case, giving Equation \eqref{timeclaim}. For the boundary conditions, we compute
	\begin{align*}
	&\left(\prescript{\alpha_1}{0} {\mathrm{R}}_u^{s_1-k'}\bigodot_{k=2}^l \prescript{\alpha_k}{0} {\mathrm{R}}_u^{s_k}\bigodot_{j=1}^{n}\left(c^{\mu_j}\prescript{\mathrm{C}}{u} {\mathrm{D}}_\infty^{\mu_j}+\lambda_{2,j}\right) \bigodot_{i=1}^m \left(c^{r_i}\prescript{\lambda_{1,i}/c}{u} {\mathrm{R}}_\infty^{r_i}\right)\right)[\phi](0)\\
	%=&\left.\prod_{i=1}^m \lambda_{1,i}^{r_i}\prod_{j=1}^n \lambda_{2,j}\prescript{\alpha_1}{0} {\mathrm{R}}_u^{s_1-k'}\bigodot_{k=2}^l \prescript{\alpha_k}{0} {R}_u^{s_k}\left(\int_{0}^{u}\phi(u-y)f_X(y)\,dy\right)\right|_{u=0}\\
	%=&\prod_{i=1}^m \lambda_{1,i}^{r_i}\prod_{j=1}^n \lambda_{2,j}\prescript{\alpha_1}{0} {R}_u^{s_1-k'}\bigodot_{k=2}^l \prescript{\alpha_k}{0} {R}_u^{s_k}\left(\phi(u)*f_{l-1}(u)*f_1(u)\left.\right)\,\right|_{u=0}\\
	=&\Lambda_{m,n}\prod_{k=2}^l\alpha_k^{s_k}\prescript{\alpha_1}{0} {\mathrm{R}}_u^{s_1-k'}\left(\phi(u)*f_1(u)\left.\right)\,\right|_{u=0},
	\end{align*}
	where 
	%$f_{l-1}$ stands for the density function of sum of $\Gamma(s_k,\alpha_k),\,k=2,\dots, l$ and 
	$f_1$ stands for the density function of $\Gamma(s_1,\alpha_1)$. Using Equation \eqref{rldci}, one has
	\begin{align*}
	&\left.\Lambda_{m,n}\prod_{k=2}^l\alpha_k^{s_k} e^{-\alpha_1 u}\prescript{}{0} {\mathrm{D}}_u^{s_1-k'}\left[\int_0^u e^{\alpha_1(u-y)}\phi(u-y) e^{\alpha_1 y}f_1(y)\,\mathrm{d}y\right]\right|_{u=0}\\
	%=&\left.\prod_{i=1}^m \lambda_{1,i}^{r_i}\prod_{j=1}^n \lambda_{2,j}\prod_{k=2}^l\alpha_k^{s_k} e^{-\alpha_1 u}\left[e^{\alpha_1 u}\phi(u)*\prescript{}{0} {D}_u^{s_1-k'}(e^{\alpha_1 u}f_1(u))\right]\right|_{u=0}\\
	%&+\prod_{i=1}^m \lambda_{1,i}^{r_i}\prod_{j=1}^n \lambda_{2,j}\prod_{k=2}^l\alpha_k^{s_k}\phi(0) \left.\prescript{}{0} {D}_y^{s_1-k'-1}\left(e^{\alpha_1 y}f_1(y)\right)\right|_{y=0}\\
	=&\Lambda_{m,n}\prod_{k=2}^l\alpha_k^{s_k}\left(\left.e^{-\alpha_1 u}\left[e^{\alpha_1 u}\phi(u)*\frac{\alpha_1^{s_1}}{\Gamma(k')}u^{k'-1}\right]\right|_{u=0}+\left.\phi(0)\frac{\alpha_1^{s_1}}{\Gamma(k'+1)}y^{k'}\right|_{y=0}\right),
	\end{align*}
	which equals to zero for $k'=1,\,\dots,\,\lceil \sum_{k=1}^l s_k \rceil-1$. This completes the proof.
\end{proof}

\subsection{The characteristic equation method}\label{subsec_cem}

Our next goal is solving the fractional differential boundary value problem in Theorem \ref{mainthm2} via a characteristic equation starting from the ansatz $ \phi(u)  = e^{-z u}$. The main technical difficulty in dealing with the full generality of Theorem \ref{mainthm2} arises from the fact that the operators in Equation \eqref{timeclaim} combine two different types of differential operators: $\mathcal{A}_{m,n}^*(c\frac{\mathrm{d}}{\mathrm{d}u}) $ is a composition of right Caputo fractional derivatives, while the operators in $ \mathcal{A}_{l}(\frac{\mathrm{d}}{\mathrm{d}u}) $ are LFDOs which are ultimately defined in terms of left Riemann-Liouville fractional derivatives (see Equation \eqref{Def_Al}, \eqref{Def_AmnStar} and \eqref{RockOp}). The proposed ansatz is an eigenfunction only for the operators in $\mathcal{A}_{m,n}^*(c\frac{\mathrm{d}}{\mathrm{d}u}) $ (see Proposition \ref{efrlfd} and Proposition \ref{effd}). When restricting to the case of $ s_k \in \mathbb{N}$ , $ k=1,\dots,l $, we simplify things greatly, since 
\begin{equation*}\label{key}
\mathcal{A}_{l}\left(\frac{\mathrm{d}}{\mathrm{d}u}\right) = \bigodot_{k=1}^l \prescript{\alpha_k}{0} {\mathrm{R}}_u^{s_k}=\bigodot_{k=1}^l\left(\frac{\mathrm{d}}{\mathrm{d}u}+\alpha_k\right)^{s_k}
\end{equation*}
reduces to a combination of ordinary differential operators. 

Note that assuming  $ s_k \in \mathbb{N}$ , $ k=1,\dots,l $ in Equation \eqref{Def_Al} is equivalent to assuming that the claim sizes  $ X_i $ are each distributed as a sum of $ l $ independent Erlang random variables. Moreover, under this case, the operator $  \mathcal{A}_{l}(\frac{\mathrm{d}}{\mathrm{d}u}) \mathcal{A}_{m,n}^*(c\frac{\mathrm{d}}{\mathrm{d}u}) $ on the left hand side of Equation \eqref{timeclaim} is a composition of right Caputo fractional derivatives. Furthermore, with the ansatz $ \phi(u) = e^{-z u} $, Equation \eqref{timeclaim} yields the following characteristic equation for $ z $:
\begin{equation}\label{cetimeclaim}
\prod_{k=1}^l (-z+\alpha_k)^{s_k}\prod_{j=1}^n(c^{\mu_j}z^{\mu_j}+\lambda_{2,j})\prod_{i=1}^m(cz+\lambda_{1,i})^{r_i}=\Lambda_{m,n}  \prod_{k=1}^l \alpha_k^{s_k}.
\end{equation}
Note that from the definition of $ \Lambda_{m,n} $ in Equation \eqref{Def_Lambda}, $ z=0 $ is always a root of \eqref{cetimeclaim}. If Equation \eqref{cetimeclaim} has $N>0$ additional distinct complex roots with positive real part, say $z_1,\dots,z_N$, then the non-ruin probability $ \phi $ that solves Equation \eqref{timeclaim} is
\begin{equation}\label{Eq_solPhi}
\phi(u) = 1+ \sum\limits_{p=1}^N K_p e^{-z_p u}, 
\end{equation}
where the constants $K_p$, $ p=1,\dots,N $ are to be determined from the boundary conditions in Equation \eqref{timeclaimBC}, which are characterized in the following result.

\begin{proposition}
	Suppose $ s_k \in \mathbb{N}$, $ k=1,\dots,l $, in Theorem \ref{mainthm2}. The number of initial-value boundary conditions of $\phi(u)$ is $N=\sum_{k=1}^l s_k$ and they are given explicitly by:
	\begin{equation}\label{maineq2bc}
	\bigodot_{k=1}^l \prescript{\alpha_k}{0} {\mathrm{R}}_u^{s_{p,k}}\bigodot_{j=1}^{n}\left(c^{\mu_j}\prescript{\mathrm{C}}{u} {\mathrm{D}}_\infty^{\mu_j}+\lambda_{2,j}\right) \bigodot_{i=1}^m \left(c^{r_i}\prescript{\lambda_{1,i}/c}{u} {\mathrm{R}}_\infty^{r_i}\right)[\phi](0)=0, \; p=1,\dots,N
	\end{equation}
	where the values of $s_{p,k}$ are to be computed as follows: let
	\begin{equation*}
	L(p)=\inf\left\{\ell\in\mathbb{N}: \sum\limits_{k=1}^\ell s_k\leqslant p\right\}, \quad p=1,\dots,N
	\end{equation*}
	and define
	\[
	s_{p,k}=
	\begin{cases}
	s_k,\quad \text{if } k<L(p),\\[10pt]
	p-\sum\limits_{i=1}^{L(p)-1}s_i-1,\quad \text{if } k=L(p),\\
	\quad\vdots\\
	0,\quad \text{if } k>L(p).
	\end{cases}
	\]
\end{proposition}
\begin{proof}
	We consider the $p$-th boundary condition
	\begin{align*}
	&\bigodot_{k=1}^l \prescript{\alpha_k}{0} {\mathrm{R}}_u^{s_{p,k}}\mathcal{A}_{m,n}^*\left(c\frac{\mathrm{d}}{\mathrm{d}u}\right)[\phi](0)\\
	%=&\left.\left(\prod_{i=1}^m \lambda_{1,i}^{r_i}\prod_{j=1}^n \lambda_{2,j}\bigodot_{k=1}^l \prescript{\alpha_k}{0} {\mathrm{R}}_u^{s_{p,k}}\left(\int_{0}^{u}\phi(u-y)\,dF_X(y)\right)\right)\right |_{u=0}\\
	=&\Lambda_{m,n}\prod_{k=1}^{L(p)-1}\alpha_k^{s_k}  \prescript{\alpha_{L(p)}}{0} {\mathrm{R}}_u^{s_{p,L(p)}}\left[\phi*f_{L(p)}*f_{L(p)+}\right](0),
	\end{align*}
	where $f_{L(p)}$ stands for the density function of a $\Gamma(s_{L(p)},\alpha_{L(p)})$ random variable and $f_{L(p)+}$ for the density function of a sum of random varibales with distributions $\Gamma(s_{k},\alpha_k),$ $k = L(p)+1,\dots, L$. Let $\Phi=\phi*f_{L(p)+}$ and apply Proposition \ref{rldci} to compute
	\begin{align*}
	\prescript{\alpha_{L(p)}}{0} {\mathrm{R}}_u^{s_{p,L(p)}}&\left[\Phi*f_{L(p)}\right](u)
	%=&e^{-\alpha_{L(p)}u}\prescript{}{0} {D}_u^{s_{p,L(p)}}\left(\int_0^u e^{\alpha_{L(p)}(u-y)}\Phi(u-y) e^{\alpha_{L(p)} y}f_{L(p)}(y)\,dy\right)\\
	=\Phi(u)\left.\prescript{}{0} {\mathrm{D}}_y^{s_{p,L(p)-1}}\left[e^{\alpha_{L(p)} y}f_{L(p)}(y) \right]\right |_{y=0}\\
	&+e^{-\alpha_{L(p)}u}\left[e^{\alpha_{L(p)}(u)}\Phi(u)* \prescript{}{0} {\mathrm{D}}_u^{s_{p,L(p)}} e^{\alpha_{L(p)} u}f_{L(p)}(u)\right].
	\end{align*}
	Note that $s_{p,L(p)-1}<s_{L(p)}$ and we have
	\begin{align*}
	&\prescript{\alpha_{L(p)}}{0} {\mathrm{R}}_u^{s_{p,L(p)}}\left[\Phi*f_{L(p)}\right](0)
	=\left.\int_0^u \Phi(u-y) \prescript{\alpha_{L(p)}}{0} {\mathrm{R}}_y^{s_{p,L(p)}}f_{L(p)}(y)\,\mathrm{d}y\,\right |_{u=0}=0.
	\end{align*}
	Since this holds for all $1\leqslant p\leqslant N$, we complete the proof.
\end{proof}

Substituting the expression in Equation \eqref{Eq_solPhi} for $ \phi(u) $ into the boundary conditions \eqref{maineq2bc} yields explicit linear equations for the unknown constants $ K_p $, $ p=1,\dots,N $. First, denote
\begin{equation*}\label{Def_Delta}
\Delta_p \defeq \prod\limits_{j=1}^n(c^{\mu_j}z_p^{\mu_j}+\lambda_{2,j})\prod\limits_{i=1}^m(cz_p+\lambda_{1,i})^{r_i}, \quad p =1,\dots,N.
\end{equation*}
Then, the constants $ K_p $, $ p=1,\dots,N $ in \eqref{Eq_solPhi} satisfy 
\begin{equation*}\label{Eq_Kp}
\begin{cases}
\mathlarger{\Lambda_{m,n}+\sum\limits_{p=1}^N\Delta_p K_p=0}\\
\mathlarger{\alpha_1\Lambda_{m,n}+\Delta\sum\limits_{p=1}^N\left(-z_p+\alpha_1\right) K_p=0}\\
%               \mathlarger{\alpha_1^2\Lambda_{m,n}+\sum\limits_{p=1}^N\Delta_p\left(-z_p+\alpha_1\right)^2 K_p=0}\\
\mathlarger{\quad\cdots}\\
\mathlarger{\alpha_1^{s_1}\Lambda_{m,n}+\sum\limits_{p=1}^N\Delta_p\left(-z_p+\alpha_1\right)^{s_1} K_p=0}\\
\mathlarger{\alpha_1^{s_1}\alpha_2\Lambda_{m,n}+\sum\limits_{p=1}^N\Delta_p\left(-z_p+\alpha_1\right)^{s_1}\left(-z_p+\alpha_2\right) K_p=0}\\
\mathlarger{\quad\cdots}\\
\mathlarger{\alpha_1^{s_1}\alpha_2^{s_2}\Lambda_{m,n}+\sum\limits_{p=1}^N\Delta_p\left(-z_p+\alpha_1\right)^{s_1}\left(-z_p+\alpha_2\right)^{s_2} K_p=0}\\ 
\mathlarger{\quad\cdots}\\
\mathlarger{\prod\limits_{k=1}^{l-1} \alpha_k^{s_k}\alpha_l^{s_l-1}\Lambda_{m,n}+\sum\limits_{p=1}^N\Delta_p\prod\limits_{k=1}^{l-1}\left(-z_p+\alpha_k\right)^{s_k}\left(-z_p+\alpha_l\right)^{s_l-1} K_p=0.}
\end{cases}
\end{equation*}

\section{Explicit Expressions for Ruin Probabilities in Gamma-time and Fractional Poisson Risk Models}\label{2exps}
The class of models considered in Theorem \ref{mainthm} is very general. In this section, we thus focus on two specific models which might be of interest to applications, and where explicit forms of ruin (non-ruin) probabilities can be derived.

\begin{remark} 
	It has been shown \cite{AsmAlb2010} that for any renewal risk model, the ruin probability always has an exponential form when the claim distribution is exponential. However, the fractional differential equation approach bridges a solid connection between classical risk model and a class of renewal models, which might be applied in a more sophisticated model. 
	%{\color{red} Moreover, this approach allows one to obtain explicit results when claim sizes are assumed to have rational Laplace transforms in such models. When claim sizes follow Gamma distribution with a non-integer shape parameter, one can still use numerical method to derive some approximations for the ruin probability.} {\color{blue} How? why? where is this shown?}
\end{remark}

\subsection{Gamma-time Risk Model}\label{gammatimerm}
A gamma-time risk model, describes the reserve process $R_r(t)$ of an insurance company by replacing the Poisson process $N(t)$ in the classical model \eqref{crm} with a renewal counting process $N_r(t)$ with $\Gamma(r,\lambda_1)$ distributed waiting times. This is a natural extension of Erlang$(n)$ risk model consiered by \cite{li2004ruin}.

Being a special case of Theorem \ref{mainthm}, the equation for ruin probability $\psi_r(u)$ in gamma-time risk model is
\begin{equation*}
%\label{fdergt}
c^r  e^{\frac{\lambda_1}{c}u}\prescript{\mathrm{C}}{u} {\mathrm{D}}_\infty^{r}\left(e^{-\frac{\lambda_1}{c}u}\psi_r(u)\right)=\lambda_1^r \left(\int_{0}^{u}\psi_r(u-y)\,\mathrm{d}F_X(y)+\int_u^\infty\,\mathrm{d}F_X(y)\right).
\end{equation*}
When claim sizes in this model have rational Laplace transforms, one could use the characteristic equation method mentioned in Section \ref{subsec_cem} to derive explicit ruin probabilities.

\begin{example}\label{Ex_GammaRM}
	In the gamma-time risk model with $\Gamma(r,\lambda_1)$ distributed inter-arrival times and Exp$(\alpha)$ distributed claim sizes, the ruin probability equals to
	\begin{equation}
	\label{grmec}
	\psi_r(u)=\left(\frac{\lambda_1}{cx_2}\right)^re^{-\left(x_2-\frac{\lambda_1}{c}\right) u},\quad u>0,
	\end{equation}
	where $x_2>\frac{\lambda_1}{c}$ is the larger root of equation
	\begin{equation*}
	\label{cegtec}
	c^r x^r\left[x-\left(\frac{\lambda_1}{c}+\alpha\right)\right]+\alpha\lambda_1^r=0.
	\end{equation*}
\end{example}

\begin{remark}
	Let $s=x_2-\frac{\lambda_1}{c}$ in Equation \eqref{grmec}, one has
	\begin{align*}
	&\left(M_X(s)M_T(-cs)\right)^{-1}-1=\left(1-\frac{s}{\alpha}\right)\left(1+\frac{cs}{\lambda_1}\right)^{r}-1\\
	%=&\frac{c^r}{\lambda_1^r}\left(\left(1-\frac{x_2-\frac{\lambda_1}{c}}{\alpha}\right)\left(\frac{\lambda_1}{c}+x_2-\frac{\lambda_1}{c}\right)^{r}-\frac{\lambda_1^r}{c^r}\right)\\
	=&\frac{c^r}{\lambda_1^r\alpha}\left[\left(\alpha+\frac{\lambda_1}{c}-x_2\right)x_2^r-\frac{\lambda_1^r}{c^r}\right]\\
	=&\frac{-1}{\lambda_1^r\alpha}\left[c^r x_2^{r+1}-\left(c^{r-1}\lambda_1+\alpha c^r\right)x_2^r+\alpha\lambda_1^r\right]=0,
	\end{align*}
	where $M_X$ and $M_T$ are moment generating functions of claim sizes and inter-arrival times. This means that $x_2-\frac{\lambda_1}{c}$ is the unique positive solution $\gamma$ of the Lundberg's fundamental equation. This finding coincides with the result from \cite{AsmAlb2010} for renewal risk models with exponential claims.
\end{remark}

In order to compare the classical compound Poisson with a gamma-time risk model, in Figure \ref{fig:test}a we show numerically computed ruin probabilities in the case of Example \ref{Ex_GammaRM} with different combinations of $ r $ and $ \lambda_1 $ when the mean claim inter-arrival time is fixed to $r/\lambda_1=1$.

%\begin{figure}[H]
%	\centering
%	\begin{subfigure}{.5\textwidth}
%		\centering
%		\includegraphics[width=6.2cm]{fig1.pdf}
%		\caption{}
%		\label{gtrp}
%	\end{subfigure}%
%	\begin{subfigure}{.5\textwidth}
%		\centering
%		\includegraphics[width=5.7cm]{fig4.pdf}
%		\caption{}
%		\label{lnu5}
%	\end{subfigure}
%	\caption{(a) Ruin probabilities in the case of Example \ref{Ex_GammaRM} for $\lambda_1=r = 0.5,\, 1,\, 1.5,\, 2$ and $2.5$. Claim sizes are taken exponentially distributed with mean $\alpha = 1$ and $c = 1.2$ in order to ensure the net profit condition. (b) Natural log of $ u_5 $ (see \eqref{Def_u5}) for the ruin probability in Example \ref{Ex_GammaRM} with continuously varying parameters $ r, \lambda_1 $. The claim sizes have fixed exponential distribution with mean $\alpha = 1$ and premium rate $c = 1.2$. The dotted line limits the region where the net profit condition $r/\lambda_1 < c$ holds (see Equation \eqref{Eq_NetProfit}).
%	}
%	\label{fig:test}
%\end{figure}

\begin{figure}[H]
	\begin{center}
		\subfloat[]      {
			\includegraphics[width=6cm]{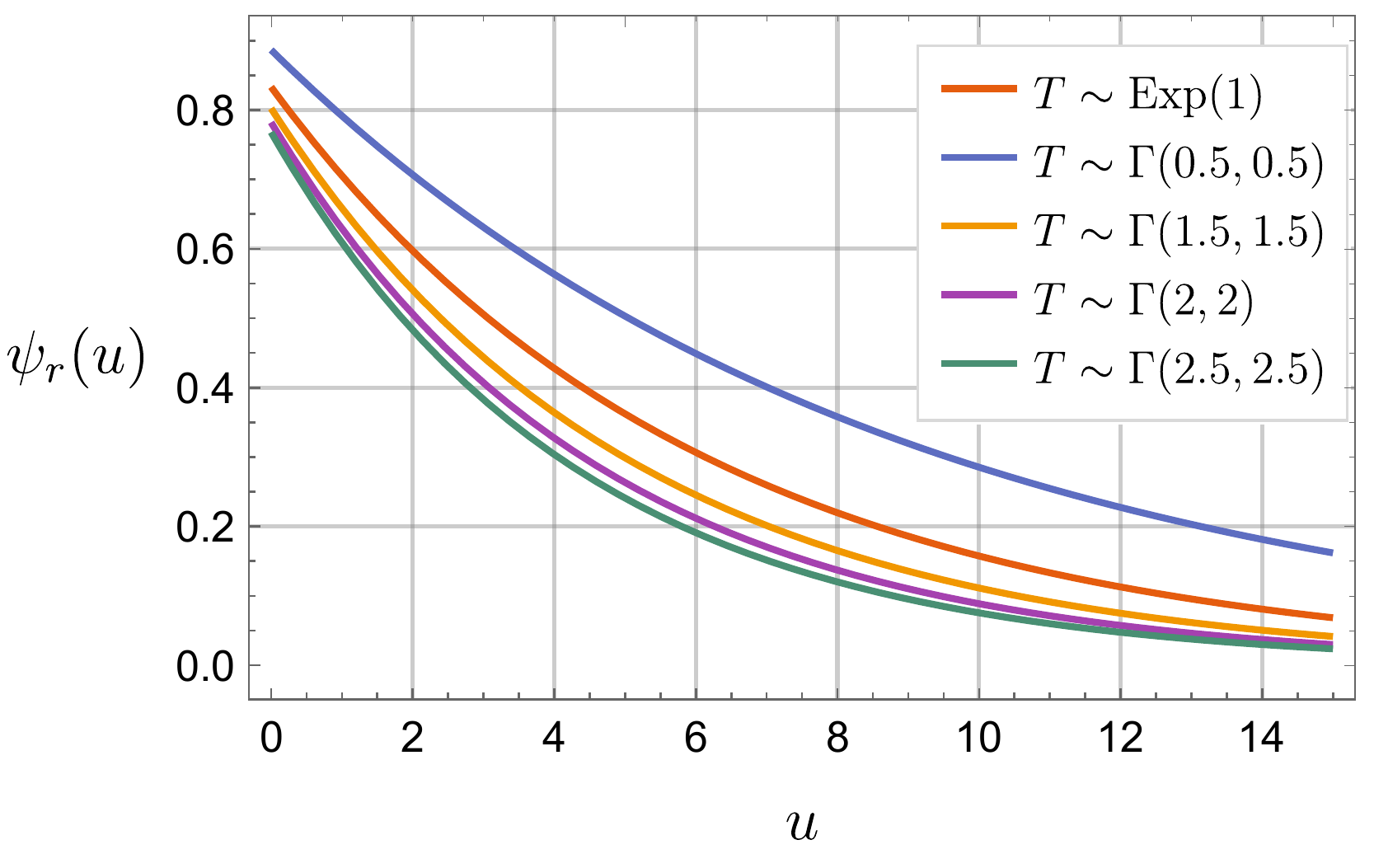}
		}\quad
		\subfloat[]      {
			\includegraphics[width=5cm]{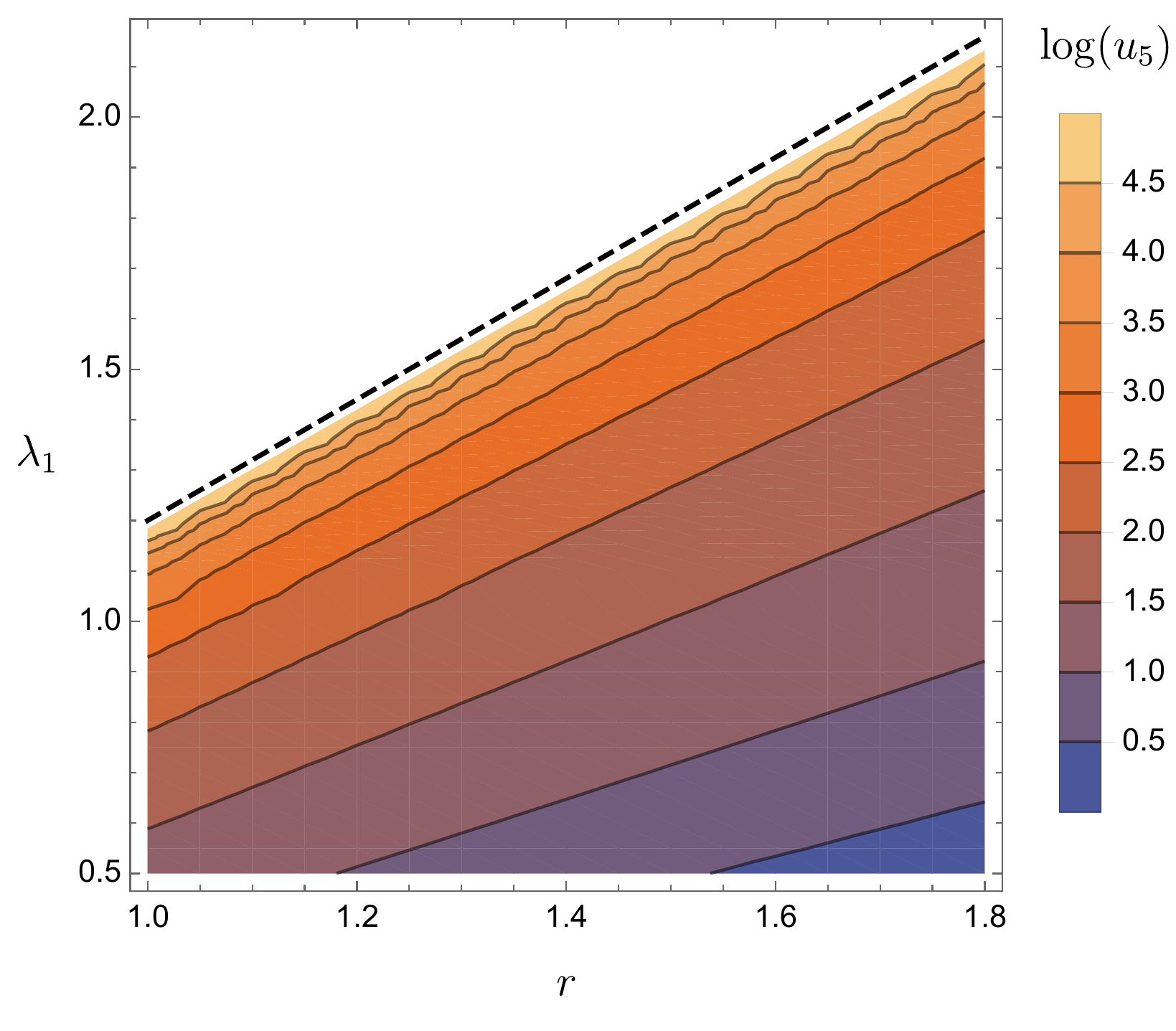}
		}
		\caption{(a) Ruin probabilities in the case of Example \ref{Ex_GammaRM} for $\lambda_1=r = 0.5,\, 1,\, 1.5,\, 2$ and $2.5$. Claim sizes are taken exponentially distributed with mean $\alpha = 1$ and $c = 1.2$ in order to ensure the net profit condition. (b) Natural log of $ u_5 $ (see \eqref{Def_u5}) for the ruin probability in Example \ref{Ex_GammaRM} with continuously varying parameters $ r, \lambda_1 $. The claim sizes have fixed exponential distribution with mean $\alpha = 1$ and premium rate $c = 1.2$. The dotted line limits the region where the net profit condition $r/\lambda_1 < c$ holds (see Equation \eqref{Eq_NetProfit}).} \label{fig:test}
	\end{center}
\end{figure}

Note the substantial impact on $\psi_r(u)$ when changing the Poisson assumption ($r = 1$). Ruin is more likely to happen in the gamma-time risk model with a larger shape parameter $r$ of inter-arrival times, and vice versa. The reason is that in this case, the expected inter-arrival time $r/\lambda_1$ is fixed whereas the variance of inter-arrival time $r/\lambda_1^2$ decreases as $r$ increases, which means that the chance of having shorter waiting periods between claims will decrease. Since ruin is usually caused by not enough capital, the model with a larger shape parameter $r$ is more likely to survive. Figure \ref{fig:test}a coincides with the finding from \cite{li2004ruin}, which focuses on Erlang($n$) risk models.

In Figure \ref{fig:test}b we illustrate the sensitivity to the parameters $ r$ and $\lambda_1 $ of the ruin probability $ \psi_r(u) $ in Example \ref{Ex_GammaRM}. In order to do this, we define the statistic
\begin{equation}\label{Def_u5}
u_5 \defeq \inf\left\{u \geq 0: \psi_r(u) < 0.05\right\}.
\end{equation}
Namely, $ u_5 $ is the minimum capital needed to ensure a ruin probability less than 5\%. Note that any combinations of $r$ and $\lambda_1$ on or above the dashed line, marking the net profit condition, will make the ruin certain. The value of $u_5$ tends to infinity as the parameters approach the dashed line since the safety loading $\frac{c\,\mathbb{E}(T)}{\mathbb{E}(X)}-1$ tends to zero. When $r$ takes large enough values or $\lambda_1$ take small enough values (in bluer areas), the ruin probability might be less than $5\%$ even with zero initial capital. Note that along contour lines, $ \mathrm{d}\lambda_1 \approx \frac{1}{c} \, \mathrm{d} r $, so the sensitivity of the ruin probabilities to its parameters depends almost exclusively on $ c $.

The next example goes a step further and assumes gamma distributions for both the inter-arrival times and the claim sizes. This case is simple enough that the two positive roots of the characteristic equation can be bounded. 

\begin{example}\label{Gammatimegammaclaim}
	In the gamma-time risk model with $\Gamma(r,\lambda_1)$ distributed inter-arrival times and $\Gamma(2,\alpha)$ distributed claim sizes, the ruin probability equals to
	\begin{equation*}
	\psi_r(u)=\frac{\frac{\lambda_1}{c}-z_3}{z_2-z_3}\left(\frac{\lambda_1}{cz_2}\right)^re^{\left(\frac{\lambda_1}{c}-z_2\right) u}+\frac{\frac{\lambda_1}{c}-z_2}{z_3-z_2}\left(\frac{\lambda_1}{cz_3}\right)^re^{\left(\frac{\lambda_1}{c}-z_3\right) u}, \quad u>0,
	\end{equation*}
	where $z_3>\frac{\lambda_1}{c}+\alpha>z_2>\frac{\lambda_1}{c}$ are the two largest roots of the equation
	$$
	c^r z^r \left[z-\left(\frac{\lambda_1}{c}+\alpha\right)\right]^2-\alpha^2 \lambda_1^r=0.
	$$
\end{example}

\subsection{Fractional Poisson Risk Model}
The fractional (compound) Poisson risk model is a special case of the classic compound Poisson risk model \eqref{crm} where the counting process is chosen as fractional Poisson process $N_\mu(t)$. Namely, the inter-arrival times are Mittag-Leffler distributed $ T \sim \text{ML}(\mu, \lambda_2) $, with $ \lambda_2 >0 $, $ 0 < \mu \leqslant 1 $. Since, when $\mu=1$, the fractional Poisson process reduces into a Poisson process, we need the net profit condition to compute the ruin probability. The following examples are derived under the assumption $0<\mu<1$ (in the fractional Poisson risk model). Note that in this case $ \mathbb{E} [T_i] = \infty $, so the net profit condition \eqref{Eq_NetProfit} holds whenever $ \mathbb{E} [X_i] < \infty $. 
It follows from Theorem \ref{mainthm} that the ruin probability $\psi_\mu$ of a fractional Poisson risk model satisfies the following fractional integro-differential equation
\begin{equation*}
%\label{FDE}
c^\mu\prescript{\mathrm{C}}{u} {\mathrm{D}}_\infty^{\mu}\psi_\mu(u)+\lambda_2\psi_\mu(u)=\lambda_2\left[\int_{0}^{u}\psi_\mu(u-y)\,\mathrm{d}F_X(y)+\int_{u}^{\infty}\mathrm{d}F_X(y)\right],
\end{equation*}
with the universal boundary condition $\lim_{u\rightarrow\infty}\psi_\mu(u)=0$. 
Explicit expressions for ruin probabilities in the fractional Poisson risk model with exponential claims has been derived in \cite{biard2014fractional}. The same result can be obtained via the fractional differential equation approach introduced in this paper. 
\begin{example}\label{Ex_FracPoisson}
	In the fractional Poisson risk model with $ T \sim \mathrm{ML}(\mu, \lambda_2) $ and exponentially distributed claim sizes with parameter $ \alpha $, the ruin probability equals
	\begin{equation*}%\label{FPPEC}
	\psi_\mu(u)=\left(1-\frac{x_2}{\alpha}\right)e^{-x_2 u},\quad u>0,
	\end{equation*}
	where $x_2$ is the unique positive solution of $
	c^\mu x-\alpha c^\mu +\lambda_2 x^{1-\mu}=0.
	$
\end{example}

Figure \ref{fig:test2}a shows the ruin probability $ \psi_{\mu}(u) $ for various combinations of the parameters $ \lambda_2, \mu $, with fixed exponential claim size distribution.

\begin{figure}[H]
	\begin{center}
		\subfloat[]      {
			\includegraphics[width=6cm]{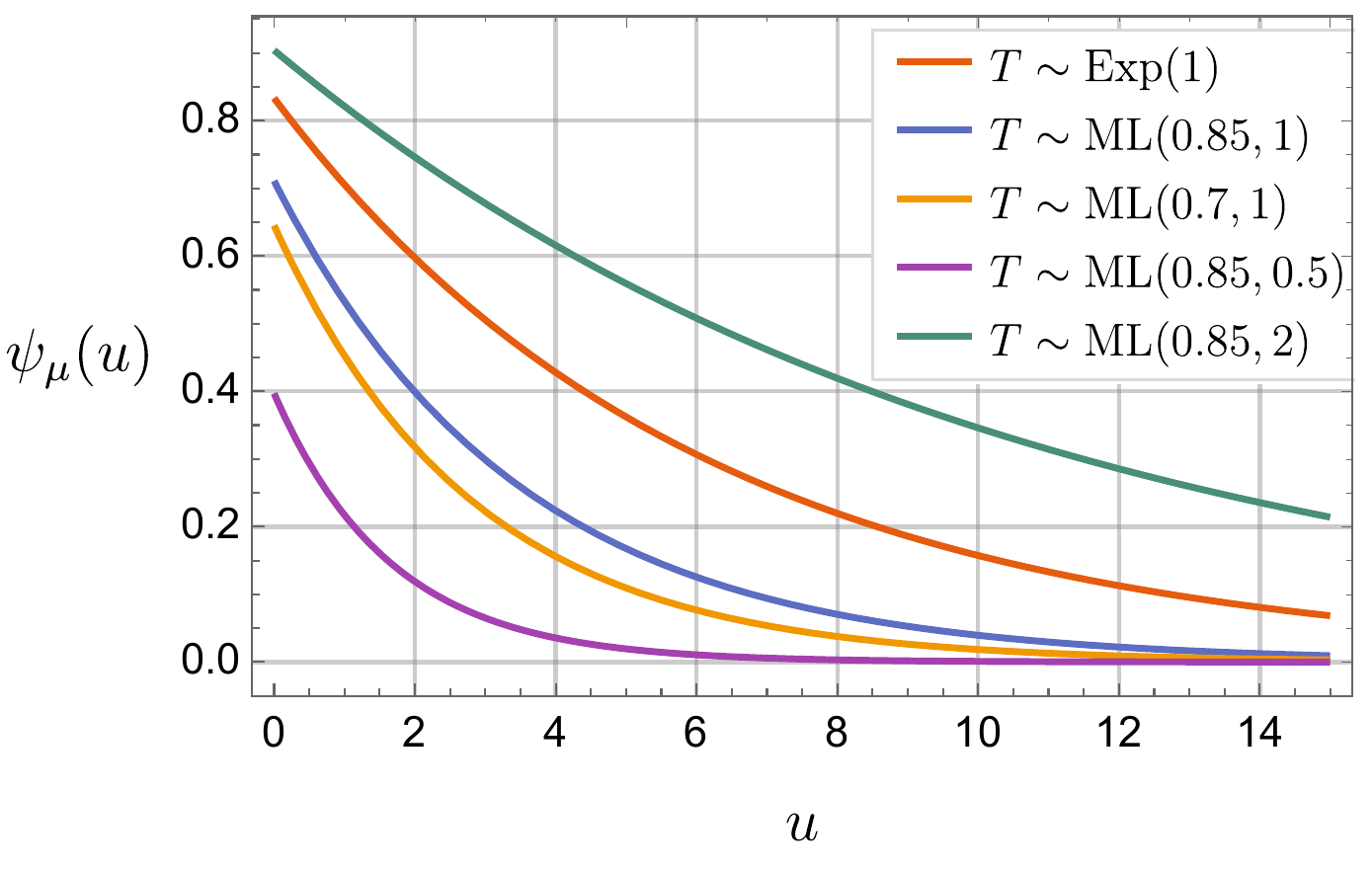}
		}\quad
		\subfloat[]      {
			\includegraphics[width=5cm]{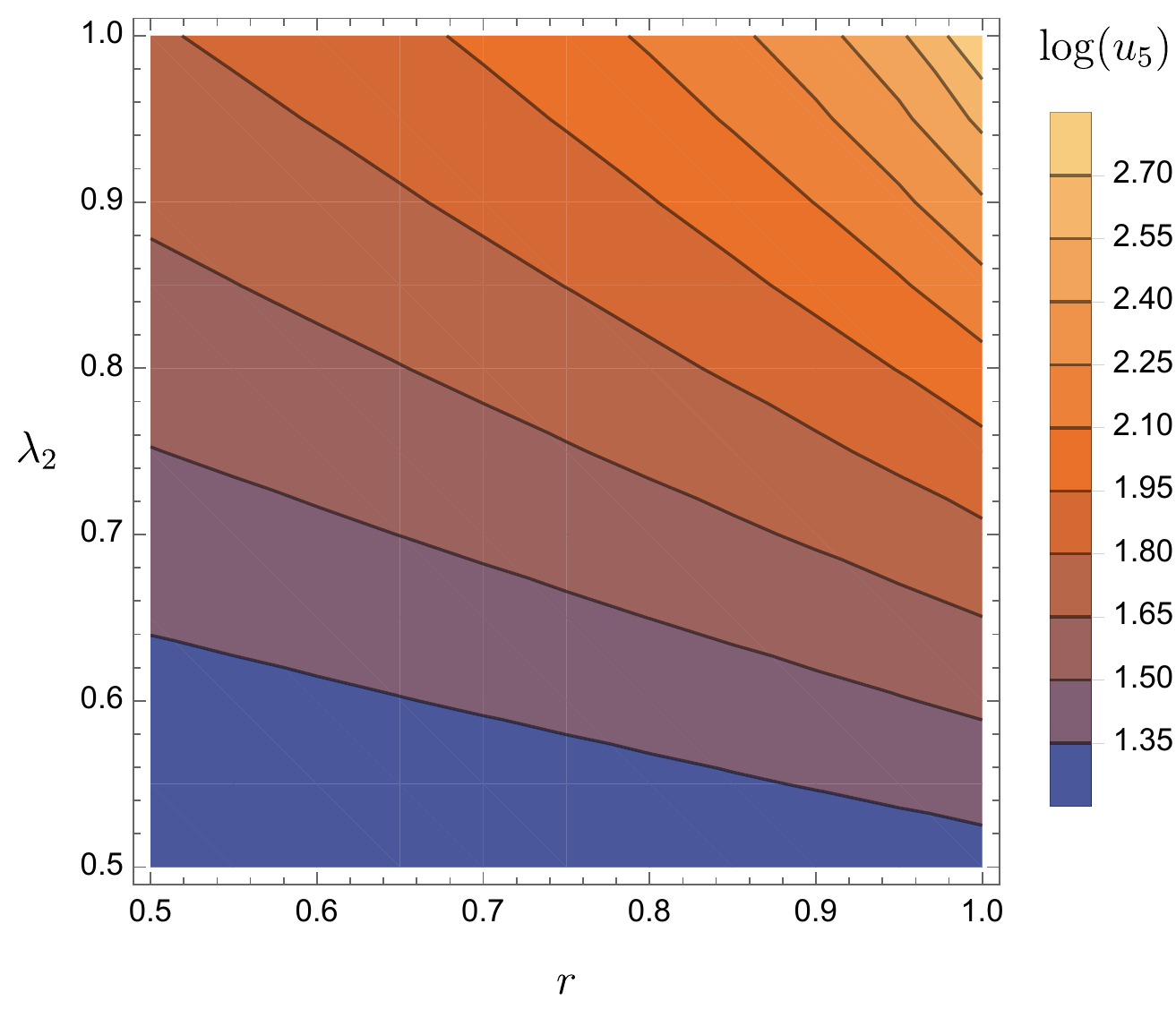}
		}
	\caption{(a) Ruin probabilities in the case of Example \ref{Ex_FracPoisson} for different combinations of $ \lambda_2, \mu $. Claim sizes are taken exponentially distributed with mean $\alpha = 1$ and $c = 1.2$. (b) Natural log of $ u_5 $ (see Equation \eqref{Def_u5}) for the ruin probability in Example \ref{Ex_FracPoisson} with continuously varying parameters $ \mu, \lambda_2 $. The claim sizes have fixed exponential distribution with mean $\alpha = 1$ and premium rate $c = 1.2$.} \label{fig:test2}
	\end{center}
\end{figure}

Note the substantial impact on $\psi_\mu(u)$ when the classical Poisson assumption ($\mu = 1$) is changed. Increasing either $\lambda_2$ or  $\mu$ increases the chances for ruin to happen. The reason is that, for large enough $ t $, the expected number of jumps, before time $t$, in the fractional Poisson process (see Equation \eqref{exnumj}), is an increasing function of both $\lambda_2$ and $ \mu $. 
Figure \ref{fig:test2}b shows the values of natural logarithm of $u_5$ as a function of $\mu$ and $\lambda_2$. Note that the contour lines in this plot are not parallel to each other. As the values of $\mu$ decrease, the parameter $\lambda_2$ plays a less significant role in the ruin probability function.

Notice that the operator $\prescript{\mathrm{C}}{u} {\mathrm{D}}_\infty^{\mu}$ tends to the identity operator, when $\mu\rightarrow 0+$. Thus, we obtain the following result.

\begin{corollary}
	In the fractional Poisson risk model, the ruin probability $\psi_\mu(u)$ converges to a function $\psi_0(u)$, as $\mu\rightarrow 0$. Moreover, the function $\psi_0(u)$ satisfies an integral equation
	\begin{equation}\label{conv}
	\left(1+\lambda_2\right)\psi_0(u)=\lambda_2\int_{0}^{u}\psi_0(u-y)\,\mathrm{d}F_X(y)+\lambda_2\int_{u}^{\infty}\mathrm{d}F_X(y),
	\end{equation}
	with the universal boundary condition $\lim_{u\rightarrow\infty}\psi_0(u)=0$. 
\end{corollary}
Substituting $u=0$ into Equation \eqref{conv} gives $\psi_0(0)=\frac{\lambda_2}{\lambda_2+1}$, which only depends on the value of $\lambda_2$. Taking Laplace transform both sides with respect to $u$ leads to
$$
\hat{\psi}_0(s)=\frac{1-\hat{f}(s)}{(\lambda_2+1)s-\lambda_2 s\hat{f}(s)},
$$
which can be explicitly inverted back in some cases.

\appendix\normalsize
\section{Basic facts from fractional calculus}\label{bkfc}
The fractional calculus is the theory of integrals and derivatives of arbitrary order, which unifies and generalizes the notions of integer-order differentiation and $n$-fold integration \cite{podlubny1998fractional}. The definitions of several special functions, fractional integrals and fractional derivatives used in this paper are listed in this section.

\subsection{Mittag-Leffler Function}
The Mittag-Leffler function was firstly introduced by \cite{mittag1903nouvelle} as a generalization of the exponential function.
\begin{definition}
	%One-parameter Mittag-Leffler function is defined on the complex plane, as a series
	%\begin{align*}
	%E_{\alpha}(z)=\sum_{k=0}^{\infty}\frac{z^k}{\Gamma(\alpha k+1)},\quad \alpha\in\mathbb{C},\,\Re(\alpha)>0,\,z\in\mathbb{C}.
	%\end{align*}
	The \emph{two-parameter Mittag-Leffler function} is defined as
	\begin{align}\label{TPML}
	E_{\alpha,\beta}(z)=\sum_{k=0}^{\infty}\frac{z^k}{\Gamma(\alpha k+\beta)},\quad \alpha,\beta\in\mathbb{C},\,\Re(\alpha)>0,\,\Re(\beta)>0,\,z\in\mathbb{C}.
	\end{align}
\end{definition}

\begin{proposition}\label{LTML}
	The Laplace transform of $z^{\alpha k+\beta-1}E^{(k)}_{\alpha,\beta}(\pm az^\alpha)$ is (see Equation (1.80) in \cite{podlubny1998fractional})
	\begin{equation*}\label{lt2ml}
	\int_{0}^{\infty}e^{-sz}z^{\alpha k+\beta-1}E^{(k)}_{\alpha,\beta}(\pm az^\alpha)\,\mathrm{d}z=\frac{k!s^{\alpha-\beta}}{(s^\alpha\mp a)^{k+1}},\quad \Re(s)>|a|^{1/\alpha}.
	\end{equation*}
\end{proposition}

\subsection{Fractional Integrals and Derivatives}
As per \cite{hilfer2008threefold}, we define and denote:
\begin{definition}\label{Def_RLI}
	The \emph{left Riemann-Liouville fractional integral} of order $r>0$ with lower limit $a \in  \mathbb{R}$ is defined on locally integrable functions $f$ as
	\begin{equation*}
	%\label{LRLFI}
	\prescript{}{a} {\mathrm{I}}^{r}_{x}f(x)=\frac{1}{\Gamma(r)}\int_{a}^{x}(x-y)^{r-1}f(y)\,\mathrm{d}y,\quad x>a,
	\end{equation*}
	and the \emph{right Riemann-Liouville fractional integral} of order $r>0$ with upper limit $b \in \mathbb{R}$ is defined as
	\begin{equation*}
	%\label{RRLFI}
	\prescript{}{x} {\mathrm{I}}^{r}_{b}f(x)=\frac{1}{\Gamma(r)}\int_{x}^{b}(y-x)^{r-1}f(y)\,\mathrm{d}y,\quad x<b.
	\end{equation*}
\end{definition}

\begin{definition}\label{Def_RLD}
	The \emph{left Riemann-Liouville fractional derivative} of order $r>0$ with lower limit $a$ is defined as the integer order derivatives of fractional integrals as follows,
	\begin{equation*}\label{lrlfd}
	\prescript{}{a} {\mathrm{D}}^{r}_{x}f(x)=\frac{1}{\Gamma(n-r)}\frac{\mathrm{d}^n}{\mathrm{d}x^n}\int_{a}^{x}(x-y)^{n-r-1}f(y)\,\mathrm{d}y,\quad x>a,
	\end{equation*}
	where $n=\lfloor r\rfloor +1$, and $\lfloor r \rfloor$ denotes the floor function. Similarly, the right Riemann-Liouville fractional derivative of order $r>0$ with upper limit $b$ is defined as
	\begin{equation*}
	\label{rrlfd}
	\prescript{}{x} {\mathrm{D}}^{r}_{b}f(x)=(-1)^n\frac{1}{\Gamma(n-r)}\frac{\mathrm{d}^n}{\mathrm{d}x^n}\int_{x}^{b}(y-x)^{n-r-1}f(y)\,\mathrm{d}y, \quad x<b.
	\end{equation*}
	These two operators are well defined on the Lebesgue space $L^{\lceil r\rceil}([a,b])$ (see Definition 2.1 in \cite{samko1993fractional}). Here, $\lceil r \rceil$ denotes the ceiling function.
\end{definition}

\begin{proposition}\label{rlsemi}
	The Riemann-Liouville fractional derivatives are the left inverse operators of the corresponding fractional integrals (see Section 3.2 in \cite{valerio2013fractional})
	\begin{equation*}
	\prescript{}{a} {\mathrm{D}}^{r}_{x} \prescript{}{a} {\mathrm{I}}^{r}_{x}f(x)=f(x)\quad \text{and}\quad
	\prescript{}{x} {\mathrm{D}}^{r}_{b} \prescript{}{x} {\mathrm{I}}^{r}_{b} f(x)=f(x), \quad \text{for any} \,\,r\in\mathbb{C}.
	\end{equation*}
\end{proposition}
%
%By using the change of variables and the property of Beta function, one has the following proposition.
\begin{proposition}\label{mlfifd}
	The left Riemann-Liouville fractional integrals $\prescript{}{a} {\mathrm{I}}^{r}_{x}$ and left fractional derivative $\prescript{}{a} {\mathrm{D}}^{r}_{x}$ of the power function $(x-a)^p$ are (see Equation (2.117) in \cite{podlubny1998fractional})
	\begin{equation*}
	%\label{fipf}
	\prescript{}{a} {\mathrm{I}}^{r}_{x}(x-a)^p=\frac{\Gamma(1+p)}{\Gamma(1+p+r)}(x-a)^{p+r}
	\,\text{and}\,
	\prescript{}{a} {\mathrm{D}}^{r}_{x}(x-a)^p=\frac{\Gamma(1+p)}{\Gamma(1+p-r)}(x-a)^{p-r}.
	\end{equation*}
\end{proposition}
%
%Since the Mittag-Leffler function is a infinite series of power function, the fractional integrals and derivatives can be computed by applying proposition \ref{mlfifd}.
%
\begin{proposition}
	The left fractional derivative $\prescript{}{0} {\mathrm{D}}^{r}_{x}$ of the two-parameter Mittag-Leffler functions satisfies (see Equation (1.82) in \cite{podlubny1998fractional})
	\begin{equation*}
	%\label{FDML}
	\prescript{}{0} {\mathrm{D}}^{r}_{x}\left[x^{\alpha k+\beta-1}E^{(k)}_{\alpha,\beta}(\lambda x^\alpha)\right]=x^{\alpha k+\beta-r-1}E^{(k)}_{\alpha,\beta-r}(\lambda x^\alpha).
	\end{equation*}
\end{proposition}
%
%The next proposition is about taking fractional derivatives of an integral depending on a parameter.
\begin{proposition}
	The left Riemann-Liouville fractional derivatives $\prescript{}{0} {\mathrm{D}}^{r}_{x}$ of an integral depending on a parameter $t \in \mathbb{R}$ is given by (see Equation (2.212) in \cite{podlubny1998fractional})
	\begin{equation*}
	%\label{rldip}
	\prescript{}{0} {\mathrm{D}}^{r}_{x} \int_0^x K(x,t)\, \mathrm{d}t=\int_0^\infty \prescript{}{t} {\mathrm{D}}^{r}_{x}K(x,t)\,\mathrm{d}t+\lim\limits_{t\rightarrow x-0}\prescript{}{t} {\mathrm{D}}^{r-1}_{x} K(x,t).
	\end{equation*}
\end{proposition}
%
%The following formula will be useful when one needs to take the left Riemann-Liouville fractional derivatives $\prescript{}{u} {D}^{r}_{x}$ on the convolution.
\begin{proposition}\label{rldci}
	The left Riemann-Liouville fractional derivatives $\prescript{}{0} {\mathrm{D}}^{r}_{x}$ of the (positive density) convolution integral equals to (see Equation (2.213) in \cite{podlubny1998fractional})
	\begin{equation*}
	\prescript{}{0} {\mathrm{D}}^{r}_{x} \left[K*f\right](x)= \left[\prescript{}{0} {\mathrm{D}}^{r}_{t}K*f\right](t)+\lim\limits_{t\rightarrow +0}f(x-t)\,\prescript{}{0} {\mathrm{D}}^{r-1}_{t}K(t).
	\end{equation*}
\end{proposition}

%{\color{red} \textbf{We need to give the domains of definition of these derivatives}}

\begin{definition}\label{Def_WLD}
	The \emph{Weyl-Liouville fractional derivatives} \cite{samko1993fractional, butzer2000introduction} are special cases of the Riemann-Liouville derivatives, whenever $a$ is replaced by $-\infty$ or $b$ is replaced by $\infty$ in Definition \ref{Def_RLD}. The right Weyl-Liouville fractional derivative is defined for functions $f\in L^{\lceil r\rceil}([a,b])$ as
	\begin{equation*}
	%\label{wrfd}
	\prescript{}{x} {\mathrm{D}}^{r}_{\infty}f(x)=(-1)^n\frac{1}{\Gamma(n-r)}\frac{\mathrm{d}^n}{\mathrm{d}x^n}\int_{x}^{\infty}(y-x)^{n-r-1}f(y)\,\mathrm{d}y, \quad n=\lfloor r \rfloor+1.
	\end{equation*}
\end{definition}
%
%\begin{proposition}The Laplace transform of ${}_{0}D^{\alpha}_{x}f(x)$ is given as \cite{Igor}
%\begin{align*}
%\mathcal{L}\{{}_{0}D^{\alpha}_{x}f(x)\}(s)&=\int_{0}^{\infty}e^{-sx}{}_{0}D^{\alpha}_{x}f(x)\,dx\\
%&=s^{\alpha}\hat{f}(s)-\sum_{k=0}^{n-1}s^k{}_{0}D^{\alpha-k-1}_{x}f(x)\Big|_{x=0}.
%\end{align*}
%\end{proposition}
%
%Another widely used definition of fractional derivatives are called Caputo fractional derivatives.

\begin{definition}\label{Def_Caputo}
	The \emph{Caputo fractional derivatives} are defined as fractional integrals on integer-order derivatives. The right Caputo fractional derivative is defined on functions $f\in L^{\lceil r\rceil}([a,b])$ as
	%\begin{equation}\label{lcrfd}
	%{}^C_{a}D^{r}_{x}f(x)=\frac{1}{\Gamma(n-r)}\int_{a}^{x}(x-y)^{n-r-1}f^{(n)}(y)\,dy,\quad x>a,
	%\end{equation}
	%and
	\begin{equation*}\label{crfd}
	\prescript{\mathrm{C}}{x} {\mathrm{D}}^{r}_{b}f(x)=\frac{1}{\Gamma(n-r)}\int_{x}^{b}(y-x)^{n-r-1}f^{(n)}(y)\,\mathrm{d}y,\quad x<b, \quad n=\lfloor r \rfloor+1.
	\end{equation*}
\end{definition}

\begin{proposition}\label{caputosemi}
	The Caputo fractional derivatives are the left inverse operators of their corresponding fractional integrals (see Section 3.2 in \cite{valerio2013fractional})
	\begin{equation*}
	\prescript{\mathrm{C}}{a} {\mathrm{D}}^{r}_{x} \prescript{}{a} {\mathrm{I}}^{r}_{x}f(x)=f(x)\quad \text{and}\quad
	\prescript{\mathrm{C}}{x} {\mathrm{D}}^{r}_{b} \prescript{}{x} {\mathrm{I}}^{r}_{b} f(x)=f(x), \quad \text{for}\,\, r\in\mathbb{N} \,\,\text{or}\,\, \Re(r)\notin \mathbb{N}.
	\end{equation*}
\end{proposition}

\begin{proposition}\label{PropFIBP}
	The Caputo and left Riemann-Liouville fractional derivatives are related by the following integration by parts formula (see Section 2.1 in \cite{almeida2011necessary})
	\begin{align}\label{FIBP}
	\nonumber\int_{a}^{b}g(x)&\prescript{\mathrm{C}}{x} {\mathrm{D}}^{r}_{b}f(x)\,\mathrm{d}x=\int_{a}^{b}f(x)\,\prescript{}{a} {\mathrm{D}}^{r}_{x}g(x)\,\mathrm{d}x\\
	+&\sum_{j=0}^{\left \lfloor{r}\right \rfloor}\left[ (-1)^{\left \lfloor{r}\right \rfloor+1+j}\left(\prescript{}{a} {\mathrm{D}}_x^{r+j-\left \lfloor{r}\right \rfloor-1}g(x)\right)\left(\prescript{}{a} {\mathrm{D}}_x^{\left \lfloor{r}\right \rfloor-j}f(x)\right)\right]^b_a.
	\end{align}

\end{proposition}
%The fractional integral by parts rules offer one of the connections between left derivatives and right derivatives.
%

\begin{proposition}\label{efrlfd}
	The eigenfunction of left fractional derivative $\prescript{}{0} {\mathrm{D}}^{r}_{x}$ (or $\prescript{\mathrm{C}}{0} {\mathrm{D}}^{r}_{x}$) is $x^{1-\alpha}E_{\alpha,\alpha}(\lambda x^\alpha)$ with eigenvalue $\lambda\in\mathbb{R}$ (see Section 2.2.3 in \cite{hilfer2008threefold}).
\end{proposition}

\begin{proposition}\label{effd}
	The eigenfunction of right fractional derivative $\prescript{}{x} {\mathrm{D}}^{r}_{\infty}$ (or $\prescript{\mathrm{C}}{x} {\mathrm{D}}^{r}_{\infty}$) is $e^{-\lambda x}$ with eigenvalue $\lambda^r$, where $\lambda\in\mathbb{R}^+$ (see Section 4 in \cite{valerio2013fractional}).
	%In fact, one has
	%\begin{align*}
	%&\prescript{C}{x} {D}^{r}_{\infty}e^{-\lambda x}=\frac{1}{\Gamma(n-r)}\int_{x}^{\infty}(y-x)^{n-r-1}(-\lambda)^ne^{-\lambda y}\,dy\\
	%=&\frac{1}{\Gamma(n-r)}\int_{0}^{\infty}z^{n-r-1}(-\lambda)^ne^{-\lambda (z+x)}\,dz=\lambda^r e^{-\lambda x}.
	%\end{align*} 
\end{proposition}
\begin{proposition}\label{ltfd}
	The Laplace transform of the left Riemann-Liouville fractional derivative of order $r>0$ is (see Equation (2.245) in \cite{podlubny1998fractional})
	\begin{equation*}
	%\label{eqltfd}
	\mathcal{L}\{ \prescript{}{0} {\mathrm{D}}^{r}_{x}f(x)\}(s)=s^r \hat{f}(s)-\sum_{k=0}^{\lfloor r\rfloor}s^k \left[ \prescript{}{0} {\mathrm{D}}^{r-k-1}_{x}f(x)\right]\big|_{x=0}
	\end{equation*}
\end{proposition}

\section{Proof of Theorem \ref{tMLGamma}}\label{appB}
\begin{proof}
	We will use induction on both variables to validate \eqref{MLGamma} together with the extra statement: for any function $g$ supported on $[0,\infty)$, $\mathcal{A}_{m,n}\left(\frac{\mathrm{d}}{\mathrm{d}t}\right)[f_{T}^{m,n}*g](t)=\Lambda_{m,n}g(t)$.
	Base step: when $m=1,n=0$ or $m=0,n=1$, from Equation \eqref{de} and \eqref{dde} we have 
	$\mathcal{A}_{1,0}(\frac{\mathrm{d}}{\mathrm{d}t})[f_T^{1,0}](t)=0$ and $\mathcal{A}_{0,1}(\frac{\mathrm{d}}{\mathrm{d}t})[f_T^{0,1}](t)=0.$
	Furthermore, a simple calculation yields
	\begin{align*}
	\mathcal{A}_{1,0}\left(\frac{\mathrm{d}}{\mathrm{d}t}\right)\left(\frac{\mathrm{d}}{\mathrm{d}t}\right)\left[f_T^{1,0}*g\right](x) &= e^{-\lambda_{1,1}t}\prescript{}{0} {\mathrm{D}}_t^{r_1}\left(e^{\lambda_{1,1}t}\left[f_T^{1,0}*g\right]\right)(t)
	=\lambda_{1,1}^{r_1} g(t),\\
	\mathcal{A}_{0,1}\left(\frac{\mathrm{d}}{\mathrm{d}t}\right)\left[f_T^{0,1}*g\right](t) &=\left(\prescript{}{0} {\mathrm{D}}_t^{\mu_1}+\lambda_{2,1}\right)\left[f_T^{0,1}*g\right](t)
	=\lambda_{2,1}  g(t).
	\end{align*}
	Inductive step: for a non-negative $m$ and $n$, we assume that the statements
	$$
	\mathcal{A}_{m,n}\left(\frac{\mathrm{d}}{\mathrm{d}t}\right)\left[f_T^{m,n}\right](t)=0, \quad \mathcal{A}_{m,n}\left(\frac{\mathrm{d}}{\mathrm{d}t}\right)[f_T^{m,n}*g](t)=\Lambda_{m,n} g(t)
	$$
	hold. We then compute,
	\begin{align*}
	\mathcal{A}_{m+1,n}\left(\frac{\mathrm{d}}{\mathrm{d}t}\right)\left[f_T^{m+1,n}\right](t)
	&=e^{-\lambda_{1,m+1}t}\prescript{}{0} {\mathrm{D}}_t^{r_{m+1}}\left[e^{\lambda_{1,m+1}t}c_{m,n} f_T^{1,0}(t)\right]=0\\
	\mathcal{A}_{m,n+1}\left(\frac{\mathrm{d}}{\mathrm{d}t}\right)\left[f_T^{m,n+1}\right](t)
	&=\left(\prescript{}{0} {\mathrm{D}}_t^{\mu_{n+1}}+\lambda_{2,n+1}\right)\left[c_{m,n} f_T^{0,1}(t)\right]=0,\\
	\mathcal{A}_{m+1,n}\left(\frac{\mathrm{d}}{\mathrm{d}t}\right)\left[f_T^{m+1,n}*g\right](t)
	&=e^{-\lambda_{1,m+1} t}\prescript{}{0} {\mathrm{D}}_t^{r_{m+1}}\left[e^{\lambda_{1,m+1}t}\,c_{m,n} f_T^{1,0}*g \right](t)\\
	&=c_{m+1,n} g(t),\\
	\mathcal{A}_{m,n+1}\left(\frac{\mathrm{d}}{\mathrm{d}t}\right)\left[f_T^{m,n+1}*g\right](t)
	&=\left(\prescript{}{0} {\mathrm{D}}_t^{\mu_{n+1}}+\lambda_{2,n+1}\right)\left[c_{m,n}  f_T^{0,1}*g\right](t)\\
	&=c_{m,n+1} g(t),
	\end{align*}
	thereby showing $m+1$ and $n+1$ cases are true. To validate the boundary conditions, we compute
	\begin{align*}
	&\prescript{}{0} {\mathrm{D}}_t^{\mu_1-k}\bigodot_{j=2}^{n}\left(\prescript{}{0} {\mathrm{D}}_t^{\mu_j}+\lambda_{2,j}\right) \bigodot_{i=1}^m \prescript{\lambda_{1,i}}{0} {\mathrm{R}}_t^{r_i}\left[f_T^{m,n-1}*f_T^{0,1}\right](0)\\
	=& \prod_{i=1}^m \lambda_{1,i}^{r_i}\prod_{j=2}^n \lambda_{2,j}\prescript{}{0} {\mathrm{D}}_t^{\mu_1-k}\left[f_T^{0,1}\right](0)
	= \prod_{i=1}^m \lambda_{1,i}^{r_i}\prod_{j=2}^n \lambda_{2,j} \lambda_{2,1} t^{k-1}E_{\mu_1,k}(-\lambda_{2,1} t^\mu_1)\big |_{t=0},
	\end{align*}
	which equals to $\Lambda_{m,n}$ when $k=1$, and $0$ for $k>1$. This completes the proof.
\end{proof}
\section{Review of Fractional Poisson Process}\label{rfpp}

The fractional Poisson process, denoted by $N_\mu(t)$, $t>0$, $\mu\in(0,1]$, is a fractional non-Markovian generalisation of Poisson process $N(t)$, $t>0$. The distribution of fractional Poisson process $P_{\mu}(n,t)=\mathbb{P}[N_\mu(t)=n]$ is defined by solving a fractional generalisation of the Kolmogorov-Feller equation \cite{laskin2003fractional}
\begin{equation*}
%\label{FKFE}
\prescript{}{0} {\mathrm{D}}^{\mu}_{t}P_{\mu}(n,t)=\lambda[P_{\mu}(n-1,t)-P_{\mu}(n,t)]+\frac{t^{-\mu}}{\Gamma(1-\mu)}\delta_{n,0},\quad t>0,
\end{equation*}
where $\lambda$ is the intensity parameter and $\delta_{n,0}$ is the Kronecker symbol.

Moreover, \cite{laskin2003fractional} showed the inter-arrival times of a fractional Poisson process have probability density function
%\begin{equation}
%\label{MLD}
$f_\mu(t)=\lambda t^{\mu-1}E_{\mu,\mu}(-\lambda t^\mu),$  $t>0.$
%\end{equation}
The Laplace transform of the inter-arrival time density $f_\mu(t)$ is
$\mathcal{L}\left\{f_\mu(t);s\right\}=\hat{f}_\mu(s)=\frac{\lambda}{s^\mu+\lambda}.$
The mean and variance of $N_\mu(t)$ are
\begin{equation}\label{exnumj}
\mathbb{E}N_\mu(t)=\frac{\lambda t^\mu}{\Gamma(\mu+1)},
\end{equation} respectively  $\Var N_\mu(t)=\frac{2(\lambda t^\mu)^2}{\Gamma(2\mu+1)}-\frac{(\lambda t^\mu)^2}{\left[\Gamma(\mu+1)\right]^2}+\frac{\lambda t^\mu}{\Gamma(\mu+1)},$ as in \cite{laskin2003fractional}.

%\begin{acknowledgements}
%\end{acknowledgements}


\bibliographystyle{spmpsci}      % mathematics and physical sciences

\bibliography{references}

\end{document}